\documentclass[10pt,a4paper,hypertex]{amsart}

\usepackage{amsmath,amsfonts,amssymb,amsthm,amscd}
\usepackage[cp1251]{inputenc}

\input xy
\xyoption{all}
\usepackage[all]{xy}
\usepackage{xspace}
\usepackage {graphicx}
\usepackage{tabulary}
\usepackage{color}
\usepackage{latexsym}
\usepackage[mathscr]{eucal}
\usepackage[mathscr]{eucal}
\usepackage[sans]{dsfont}
\usepackage{stmaryrd}
\usepackage{enumerate}

\usepackage[T1]{fontenc}
\usepackage[active]{srcltx}
\usepackage{a4wide}
\usepackage{ifthen}

\usepackage[hyperindex,bookmarksnumbered,plainpages,
colorlinks=true,
urlcolor=green,frenchlinks,final=true]{hyperref}

\definecolor{darktangerine}{rgb}{1.0,0.66,0.07}
\definecolor{red}{rgb}{0.0,0.0,0.0}

\makeatletter

\DeclareMathOperator{\Aut}{Aut}
\DeclareMathOperator{\Mat}{Mat}
\DeclareMathOperator{\Word}{\mathtt{Word}}
\DeclareMathOperator{\gl}{gl}
\DeclareMathOperator{\id}{id}
\DeclareMathOperator\Quot{Quot}
\DeclareMathOperator\sym{Sym}
\DeclareMathOperator{\Mon}{\mathtt{Mon}}
\DeclareMathOperator{\SU}{SU}
\DeclareMathOperator{\mdeg}{mdeg}
\DeclareMathOperator{\tdeg}{tdeg}

\def\sl{{\operatorname{sl}}}
\def\su{{\operatorname{su}}}
\def\so{{\operatorname{so}}}

\newcommand{\frc}[2]{\raisebox{2pt}{\text{\footnotesize #1}}/\raisebox{-3pt}{\text{\footnotesize #2}}}
\newcommand\ot{\frc{1\!}{\!2}}  

\numberwithin{equation}{section}
\newtheorem{theorem}{Theorem}
\newtheorem{lemma}{Lemma}[section]
\newtheorem{definition}[lemma]{Definition}
\newtheorem{corollary}[lemma]{Corollary}
\newtheorem{remark}[lemma]{Remark}
\newtheorem{denotations}[lemma]{Definition-Notation}
\newtheorem{proposition}[lemma]{Proposition}
\newtheorem{step}[lemma]{Step}
\newtheorem{assumption}[lemma]{Assumption}

\DeclareSymbolFont{greek}{U}{eur}{m}{n}
\SetSymbolFont{greek}{bold}{U}{eur}{b}{n}
\DeclareSymbolFontAlphabet{\gr}{greek}
\DeclareMathSymbol{\UpR}{\mathord}{greek}{`R}
\DeclareMathSymbol{\glambda}{\mathord}{greek}{"15}
\DeclareMathSymbol{\vmu}{\mathord}{greek}{"16}
\DeclareMathSymbol{\vphi}{\mathord}{greek}{"1E}
\DeclareMathSymbol{\vpsi}{\mathord}{greek}{"20}
\DeclareMathSymbol{\vrho}{\mathord}{greek}{"1A}
\DeclareMathSymbol{\vsigma}{\mathord}{greek}{"1B}
\DeclareMathSymbol{\vtau}{\mathord}{greek}{"1C}
\DeclareMathSymbol{\gzeta}{\mathord}{greek}{"10}
\DeclareMathSymbol{\vtheta}{\mathord}{greek}{"12}
\DeclareMathSymbol{\gnu}{\mathord}{greek}{"17}
\DeclareMathSymbol{\uii}{\mathord}{greek}{"69}

\DeclareMathSymbol{\gF}{\mathord}{greek}{`F}
\DeclareMathSymbol{\bG}{\mathord}{greek}{`G}
\DeclareMathSymbol{\gK}{\mathord}{greek}{`K}
\DeclareMathSymbol{\gL}{\mathord}{greek}{`L}
\DeclareMathSymbol{\gM}{\mathord}{greek}{`M}
\DeclareMathSymbol{\ue}{\mathord}{greek}{`e}
\DeclareMathSymbol{\gp}{\mathord}{greek}{`p}
\DeclareMathSymbol{\gq}{\mathord}{greek}{`q}
\DeclareMathSymbol{\gb}{\mathord}{greek}{`b}
\DeclareMathSymbol{\vv}{\mathord}{greek}{`v}
\DeclareMathSymbol{\uu}{\mathord}{greek}{`u}
\DeclareMathSymbol{\ww}{\mathord}{greek}{`w}
\DeclareMathSymbol{\gS}{\mathord}{greek}{`S}

\def\gX{\cU^{\text{x}\,}}
\def\gY{\cU^{\text{y}\,}}
\def\gH{H}
\def\ux{X}
\def\uy{Y}
\def\uz{Z}

\def\bk{\Bbbk}
\def\bgM{\,\,\overline{\!\gM\!\!\!\!\phantom{|}}\,\,}
\def\LLt{\gL^{\!^{_\times}}}
\def\LL{\gL\,}

\def\bww{\overline{\ww\!}\,}

\newcommand\bS{\mathfrak{S}}\newcommand\op{\operatorname}

\newcommand\bC{\mathbb{C}}
\newcommand\bR{\mathbb{R}}
\newcommand\bZ{\mathbb{Z}}
\newcommand\Si{\Sigma}
\newcommand\Om{\Omega}

\newcommand\La{\Lambda}
\newcommand\cU{\mathcal{U}}

\newcommand\bL{\mathrm{L}}
\newcommand\fA{\mathfrak{A}}
\newcommand\Alg{\mathcal{A}\text{lg}\,}

\newcommand\ssi{\sigma}
\newcommand\ga{\gamma}
\newcommand\la{\lambda}

\newcommand\bbla{\bar{\lambda}}

\def\cC{\mathcal{C}}
\def\cH{\mathcal{H}}
\def\cX{\mathcal{X}}
\def\cY{\mathcal{Y}}
\def\cZ{\mathcal{Z}}
\def\cV{\mathcal{V}}
\def\ff{f}

\def\tki{{ki}}
\def\tkj{{kj}}
\def\tti{{2i}}
\def\ttj{{2\!j}}
\def\tmj{{mj}}
\def\ft{\mathfrak{t}}
\def\fg{\mathfrak{g}}
\def\Sm{\mathcal{S}}
\def\tt{\mathtt t}
\def\bt{\bar{\mathtt t}}
\def\muxx{\mu^{xx}}
\def\muxy{\mu^{xy}}
\def\muyx{\mu^{yx}}
\def\muyy{\mu^{yy}}

\def\J{J}
\def\Jp{\mathbb{J}_{\gp}}
\def\Jsp{\mathbb{J}_{\!\gp}}
\def\Jq{\mathbb{J}_{\gq}}
\def\Jsq{\mathbb{J}_{\!\gq}}

\def\rg{\UpR}
\def\brg{\bar{\UpR}}
\def\ttd{\mathrm{d}}
\newcommand\is{_{i}}
\newcommand\js{_{\!j}}
\newcommand\ks{_{k}}
\def\psub{_{(\!\gp)}}
\def\qsub{_{(\!\gq)}}

\def\un{\mathtt{1}}
\def\st{^{\!*}}
\def\mat{$\bullet$\,\,}
\def\vcup{\overset{\to}{\,\bigsqcup\,}}
\def\mo{^{\!-\!1}}
\def\le{{[\hspace{-0.15cm}[}\,}
\def\ra{{\;]\hspace{-0.15cm}]}\,}

\def\osum{\mathop\sum\limits}
\def\ocup{\mathop\cup\limits}
\def\obigoplus{\mathop\bigoplus\limits}
\def\bigocup{\mathop\bigcup\limits}
\def\oprod{\mathop\prod\limits}

\begin{document}

\author{Natalia Golovashchuk and Jo\~{a}o Schwarz}
\title[On rational twisted generalized Weyl algebra \ ] {On  rational twisted generalized Weyl algebra}

%

\begin{abstract}
The aim of this work is to investigate the structure of some skew twisted algebras,
when the coefficient ring is a localization of the polynomial ring over the field of characteristic zero, and an involution is provided.
A parallel construction of the rational twisted generalized Weyl algebras is given.
We propose a method and explicit formulas for a constructive description of these algebras
and their involution-symmetric invariant subalgebras based on the Gelfand-Zeitlin
realization of the universal enveloping algebra of some complex Lie algebras.
As concrete examples we discuss special unitary and orthogonal algebras of rank three.
\end{abstract}

\maketitle


Subjects: Rings and Algebras (math.RA); GWA Algebra (math.QA) 


\tableofcontents

\section*{Preliminaries}

\smallskip~

The concept of twisted generalized Weyl algebras (TGWA), introduced by Mazorchuk and Turowska in \cite{MT-99}
and additionally studied in \cite{MT-99}, \cite{MPT}, \cite{H1}, \cite{FH-1}, \cite{HO}, \cite{HS}, appeared as a natural class of non-commutative polynomial algebras in some set of indeterminates over a commutative polynomial ring, and the multiplication acts with some additional factors from the coefficient ring.
In further work on this topic, the polynomial ring was replaced by some localization of the polynomial ring or its field of quotients, which naturally leads to the concept of rational algebra. Moving further in this direction, it is natural to extend the range of values of the basic parameters of the algebras under consideration to some rational extensions of the base coefficient ring, all the more since such an action turns out to be necessary from the point of view of the enveloping algebras of some classical Lie algebras in order to represent them as rational TGW algebras.

We introduce the concept of rational TGW algebra and propose a constructive description of the basis of universal enveloping algebra for some Lie algebras based on the Gelfand-Zeitlin formulas (in Zhelobenko form, see \cite{zh:cg}) given for the action of generators of semi simple Lie algebras on a Gelfand-Zeitlin basis of a finite dimensional  irreducible representations, see \cite{Z}, \cite{M}, \cite{Ser}, \cite{Maz-ort}.

Futorny and Ovsienko showed in \cite{FO1} that the Gelfand-Tsetlin formulas for $\gl_n$ define an embedding
of the corresponding  universal enveloping algebra into a skew group algebra of a free abelian group
over some field of rational functions. Besides, it was shown in \cite{FH} that the consistency relation in TGWA can be expressed as identities in certain localization of the coefficient ring. In this paper, we give a slightly different, but very close in spirit result, assuming that the basic parameters of the considered algebra belong to a localization of the coefficient ring.

We consider algebras that are some modification of TGW algebras in two ways.
We assume that the free variables commute up to some localization of the coefficient polynomial ring.
Moreover, these algebras are endowed with involution both on the base coefficient ring and on free noncommuting variables of the algebra.
The subgroup of symmetric group acts naturally on the determined algebra according to some given partition of the set of indeterminates; we investigate the invariant subalgebras.

More precisely, we define a natural class of $\gS$-rational twisted generalized Weyl algebras under the assumption that the commutative coefficient ring is a localization of the polynomial ring with respect to some multiplicatively closed subset $\gS$. We call such algebras rational or RTGW algebras.
We show that those algebras are constructed using the concept of skew twisted monoid algebra equipped with an involution,
and are some modifications of the well-studied twisted generalized Weyl algebras.
Furthermore, we show for concrete examples of special unitary and orthogonal Lie algebras of rank three
that there exists an injective homomorphism from the universal enveloping algebra $\op{U}(\fg)$
to the RTGW algebra $\fA$ that induces an injective homomorphism of the corresponding Lie algebras (Theorems \ref{the_su_RTGWA}, \ref{the_so_RTGWA}).

The structure of the paper is as follows.
First we give some constructions of skew twisted monoid algebras endowed with involution in full generality.
Then we narrow down this definition to the case of saturated, involuted, double localization of the coefficient ring,
and we construct in Section 2 a twisted generalized Weyl algebra over some localization of the polynomial main ring
and we study some basic facts about these algebras.

After that, in Section 3 we construct some examples of RTGW algebras, specifying directly their defining parameters,
and show, by direct computation, that some their invariant subalgebras, namely $\fA_\su$ and $\fA_\so$ are isomorphic to
universal embedding algebras of the special unitary and orthogonal Lie algebras of rank $3$, respectively.

\vspace{0.1cm}

In addition to its own interest, the usefulness of the construction is reflected in the following Theorems presented in the last section, which describe two specific applications of this theory.

\vspace{0.1cm}

\begin{theorem}\label{the_su_RTGWA}
Let $\fA=\Alg_{3,6}\big(\gS,\bG,\vsigma,\ft,\vmu\big)$ be the complex RTGW algebra that satisfies Definition \eqref{def_RTGWA_general}
and has the parameters defined by \eqref{def_tt_bt_uni}, \eqref{equ_mu_tt_dependence_uni}.
Then there exists an injective $\bC$-linear homomorphism
$$\vpsi:\op{U}(\su(3;\bC))\longrightarrow\fA^{\!\bG}$$
from the universal enveloping algebra of complex special unitary Lie algebra of rank $3$ to $\bG$-invariant subalgebra of $\fA$,
such that $\vpsi$ induces an isomorphism of the Lie algebras
$\vpsi:\su(3;\bC)\,\,\overset{\sim}{\longrightarrow}\,\,\fg_\su(3;\bC)$,
where $\fg_\su(3;\bC)$ is a Lie subalgebra of the Lie algebra of $\fA$,
and $\op{U}(\fg_\su(3;\bC))\subset\fA^{\!\bG}$ is an invariant subalgebra with respect to the group $\bG$ action.
\end{theorem}

\vspace{0.1cm}

\begin{theorem}\label{the_so_RTGWA}
Let $\fA=\Alg_{2,4}\big(\gS,\bG,\vsigma,\ft,\vmu\big)$ be the complex RTGW algebra,
satisfying Definition \eqref{def_RTGWA_general}, and having the parameters given by definition \ref{def_rg_ort_formulae}.
Then there exists an injective $\bC$-linear homomorphism $$\vpsi:\op{U}(\so(3;\bC))\longrightarrow\fA^{\!\bG}$$
from the universal enveloping algebra of complex orthogonal Lie algebra of rank $3$ to to the invariant subalgebra of $\fA$,
which gives an isomorphism of Lie algebras $\vpsi:\so(3;\bC)\,\,\overset{\sim}{\longrightarrow}\,\,\fg_\so(3;\bC)$,
where $\fg_\so(3;\bC)$ is a subalgebra of the Lie algebra of $\fA$ defined in \eqref{formula_GZ-3-delta-full}, in addition $\op{U}(\fg_\so(3;\bC))\subset\fA^{\!\bG}$.
\end{theorem}

Therefore, the universal enveloping algebra of a special unitary or of a special orthogonal algebra of rank $3$ can also be realized
as an invariant subalgebra of a Lie subalgebra of RTGWA of the same rank.

\vspace{0.1cm}

\section{A general setting}

\vspace{0.1cm}

This section contains some basic material that will be used through this paper.

Let $\bZ_+$ be a set of all strictly positive integers.
The interval of integers $i$ with $a\leqslant i\leqslant b$ is denoted by $\le a;b \ra$ or by $[a;b]$.
The product of sets always means the Cartesian product and is denoted by $\times$.

The ground field $\bk$ is commutative of characteristic zero containing field of real numbers $\bR$.
In the specific examples of unitary and orthogonal algebras considered in the last section, we fix the field $\bk=\bC$.
Term \emph{ring} will always denote an associative unital $\bk$-ring that is an integral domain.

%

By involution of the $\bk$-ring $\bR$ we mean the self inverse $\bR$-automorphism,
acting on $\bk$ either identically or by complex conjugation in the case $\bk=\bC$.

\vspace{-0.1cm}


\vspace{0.1cm}

\subsection{Polynomial ring and shift automorphisms}

\smallskip~

We fix the positive integer $\gp\in\bZ_+$.
Assume, we are given by a $\gp$-tuple of indeterminates $\glambda=\glambda\psub=\langle \la\is \rangle_{i\in\,\Jp}$,\,
where $\Jsp= {\le{1;\gp}\ra}$ be an index set.
Let $\La=\bk[\,\glambda\,]$ denotes the polynomial ring on $\glambda$, $\Aut(\La)$ be the automorphism group,
$\Quot(\La)$ the quotient algebra of $\La$. We put $\un_\La=\un_\bk=\un$.

Let $\La$ be endowed with an involution $*$ associated with the $\gp$-tuple $\gzeta=\gzeta\psub=\langle \zeta\is \rangle_{i\in\Jp}$
with $\zeta\is\in\{0,1\}$,\, such that $\la\is^*=\zeta\is-\la\is$ for any $i\in\Jsp$.
We assume $\La$ to be a \emph{polynomial ring with involution} denoted by $\La{=}\bk[\,\glambda\psub]^{\gzeta}$
or simply by $\La{=}\bk[\,\glambda]$ if all parameters are fixed.
For convenience, in what follows, we introduce the notation: $\bbla\is:=\la\is^*$ for any $i\in\Jsp$.

\vspace{0.1cm}

We assume, there is a commutative shift group $\Sigma\subset\Aut_\bk(\La)$,
$\Sigma=\langle\, \ssi\is \rangle_{i\in\Jp}\simeq\bZ^{\otimes{\gp}}$
generated by the left one step shift automorphisms $\ssi\is\in\Aut_\bk(\La)$ with respect to $\la\is$, $i\in\Jsp$,
so that $\ssi\is(\la\is)=\la\is-1$, and $\ssi\is(\la\js)=\la\js$ if $j\ne i$.

\subsection{The involuted polynomial ring and localization}

\smallskip~

\begin{definition}\label{def_isd_mult_set}

The multiplicatively closed set $\gS$ belonging to the polynomial ring with involution $\La$ is called:

\emph{involuted} if $s^*\in\gS$ whenever $s\in\gS$;

\emph{saturated} if $\pm1\in\gS$, and if it is closed with respect to taking non-constant divisors of each of its elements;

\emph{double} if each indecomposable non-constant divisor of any $s\in\gS$ is a linear polynomial the form $\pm\la\is+c$, $\pm2\la\is+c$, $\pm\la\is\pm\la\js+c$, or $c\in\bZ$, $i,j\in\Jsp$
(up to factors belonging to $\bk$).
\end{definition}


\begin{definition}[of stable multiplicative set]\label{stable multiplicative set}
The set $\gS$ is called $*$,$\Si$-stable multiplicative set if $\ssi(s)\in\gS$ and $s^*\in\gS$ for any $s\in\gS$ and any $\ssi\in\Si$.

We say that $\gS$ is a \emph{multiplicative closure} of the subset $\Om\subset\gS$
if each element of $\gS$ is a finite product of the elements from $\Om\cup\{\pm1\}$;
we write in this case $\gS=\gS(\Om)$.

We say that $\gS$ is finitely generated by the subset $\Om_0\subset\gS$ if $|\Om_0|<\infty$ and $\gS=\gS(\Om_0^{\Si})$
where $\Om_0^{\Si}=\{\ssi(s) \mid s\in\Om_0,\, \ssi\in\Si\}\cup\{\pm1\}$.
\end{definition}

\vspace{0.1cm}

The localization of $\La$ at $\gS$ is denoted by $\bL=\gS^{\!-\!1}\La$.
We have $\Sigma\subset\Aut_\bk(\LL)$ with $\ssi(f/g)=\ssi(f)/{\ssi(g)}$.
The involution map $*:\La\to\La$ is extended to the map $*:\LL\to\LL$ with $(f/g)^*=f^*/{g^*}$, because $*:\gS\to\gS$. $\LLt\subset\Quot(\La)$ denotes the group of units of $\LL=\gS^{\!-\!1} \Lambda$.

We assume that the symmetric group $\sym_\gp$ acts on the sets of variables $\glambda\psub$ and $\zeta\psub$ by permutation of indices,
in what follows, we will assume that $\operatorname {Sym}_\gp$is a subgroup of $\Aut_\bk(\LL)$.

\vspace{0.1cm}

\begin{definition}\label{def_diskr_def}
Let $\La{=}\bk[\,\glambda\psub]^{\gzeta}$ be a polynomial ring with involution.
For $i,j{\in}\Jsp$, $i{\ne}j$, take
\begin{align*}
\begin{array}{lll}
\ttd\is &=& \la\is\bbla\is\,(\la\is-\bbla\is)(\bbla\is-\la\is);
\\
\ttd_{i,j} &=& (\la\is+\la\js)(\la\is+\bbla\js)(\bbla\is+\la\js)(\bbla\is+\bbla\js);
\\
\ttd &=& \oprod_{i\in\Jp}\ttd\is \,\cdot\! \oprod_{i,j\in\Jp,i{<}j} \ttd_{i,j}.
\end{array}
\end{align*}
The polynomial $\ttd\in\La$ is called the \emph{defining polynomial} for the localization of $\La$.
\end{definition}

\begin{remark}\label{rem_Om_prime_diskr}
Let $\Om_{0}$ be the set of all non scalar indecomposable dividers (of total degree $1$) of the defining polynomial $\ttd$,
and let $\gS_{\Om}=\gS(\Om_{0}^{\Si})$ be a multiplicative set  generated by the countable set $\Om_0^{\Si}$
(see definition \ref{stable multiplicative set}).
Then $\gS_{\Om}$ is the saturated, involuted, double multiplicative set.

Moreover,  $\gS_{\Om}$ coincides with the $\Si$-stable multiplicative closure of a finite set
$$\bigcup_{i\in\,\Jp} \big\{ \la\is, 2\la\is-1 \big\} \,\cup \,
\bigcup_{i,j\in\,\Jp;\, i<j} \big\{\la\is+\la\js,\,\, \la\is{-}\la\js \big\} \,\cup \, \{\pm1\}.$$
\end{remark}

\begin{proof}
Indeed, for $i\ne{j}$ we have:\; $\bbla\is=-\ssi\is^{\zeta\is}(\la\is)$, $\ssi\is(\la\js)=\la\js$, and therefore
$\la\is\pm\bbla\js=\ssi\js^{\zeta\js}(\la\is\mp\la\js)$,\;\; and
$(\bbla\is\pm\bbla\js)=-\ssi\is^{\zeta\is}\ssi\js^{\zeta\js}(\la\is\pm\la\js)$.
\end{proof}

\begin{definition}\label{def_rational involuted ring}
Having a polynomial ring with involution $\La=\bk[\,\glambda\psub]^{\gzeta}$ and the multiplicative set $\gS=\gS_{\Om}$,
we say that the localization $\gS^{{-}1}\!\La$ is linear, involuted and double,
and the ring $\LL=\gS^{{-}1}\!\La$ is called a \emph{rational involuted ring}.
\end{definition}




\subsection{Storey partition}

\smallskip~
\smallskip~

\begin{definition}[storey partition $\&$ partial localization]\label{def_storey_partition}
Let $\gp,\gq\in\bZ_+$, $0{<}\gq{\leqslant}{\gp}$.
The decomposition of the index set
\begin{align}\begin{split}\label{equ_storey_partition}
\Jsp={\le{1;\gp}\ra}=\Jq\cup{J}_{\!r{+}1},\;\; {J}_{\!r{+}1}={\le{\gq{+}1,\gp}\ra},
\\
\Jsq={\le{1;\gq}\ra}= {{J}_{\!1}} \cup\ldots\cup\,{{J}_{r}}
={\le{1;\gq_1}\ra}\cup\, {\le{\gq_1{+}1;\gq_2}\ra} \cup\ldots\cup\,{\le{\gq_{r{-}1}{+}1,\gq_r}\ra},\;\;\; \gq_r=\gq,
\end{split}\end{align}
is called the $\gp,\gq$ \emph{storey partition} (of length $r{+}1$).

\vspace{0.1cm}

Given the storey partition $\Jsp= {{J}_{\!1}} \cup\ldots\cup\,{{J}_{r{+}1}}$ we denote by $\vec{\ttd}$ the product of all
$\ttd_{i}$, $i\in\Jsp$, and all $\ttd_{i,j}$, $i\in J_k$, $j\in{J}_m$ with $|k-m|\leqslant{1}$, $1\leqslant{k,m}\leqslant{r{+}1}$.
Then polynomial $\vec{\ttd}$ divides polynomial $\ttd$, and $\vec{\ttd}=\ttd$ for the case $r=0,1$.

Consider a subset $\Om'_{0}\subset\Om_{0}$ containing all dividers of $\vec{\ttd}$.
The corresponding  multiplicative closure $\gS'=\gS(\Si(\Om'_{0}))$ is a multiplicative set corresponding
the storey partition of $\Jsp$, it is saturated, involuted and double.
In this case the set $\gS'=\gS(\Si(\Om'_{0}))$ is called the partial multiplicative set
and the localization $\bL=(\gS')^{\!-\!1}\La$ is called a \emph{partial localization}.
The map $*$ is naturally extended to an involution $*\in\Aut(\LL)$ and satisfies the condition $*:\LLt\to\LLt$.
\end{definition}

\vspace{0.1cm}

For any set $I=\le{a;b}\ra$ we denote by  $\sym_I$ the symmetric group on the set of indices $I$. For a ring $\La{=}\bk[\,\glambda\psub]^{\gzeta}$, a subgroup $\bS\subset\sym_\gp$
is called compatible with storey partition \eqref{equ_storey_partition} if
$\zeta_i=\zeta_j$ for $i,j\in J_k$,\; $1\leqslant{k}\leqslant{r{+}1}$;\;
and $\bS=\bS_{\!1}\times\bS_{2}\times \ldots \times \bS_{r{+}1}\subset\Aut_\bk(\La)$ with $\bS_{k}=\sym_{J_k}$.
So $\bS$ is a direct product of symmetric groups on the subsets of indices ${J}_{\!1}, \ldots {J}_{r{+}1}$ corresponding to the partition.

The elements of the group $\bS$ act on generators $\glambda$ of the ring $\La$ and on the parameters $\gzeta$ simultaneously.
Under this, involution $*$ commutes with the permutations from $\bS$; we denote by $\bG$ the group $\bG=\bS\times\langle*\rangle$.

\vspace{0.1cm}

\begin{assumption}\label{ass_mult_set}
Later in the text, when talking about a multiplicative set over the polynomial ring with involution $\La$, we will mean exclusively
the $\Si$-stable saturated, involuted, double multiplicative set $\gS$,
the word localization will mean only the localization defined above.
Moreover, we will consider only those ring automorphisms $\phi\in\Aut(\La)$ for which condition $\phi(\gS) \subset \gS$ holds.
In particular, groups $\Si$ and $\sym_\gp$ are embedded in $\Aut(\La)$.

In the case of a storey partition of the set of indices $\Jq$ compatible with ring $\La$ and subgroup $\bS$, we assume the existence
of a partial localization according to the storey partition; moreover, the partial multiplicative set
is a subset of the saturated, involuted, double multiplicative set of $\La$.
\end{assumption}

\vspace{0.1cm}

\subsection{Words and bimonomials}

\smallskip~
\smallskip~

Given two sets of indeterminates $\cX=\{\ux_i\}_{i\in\Jsq}$, $\cY=\{\uy_i\}_{i\in\Jq}$ with $\cX\cap\cY=\varnothing$
for $\Jsq=\le{1;\gq}\ra$, $\gq\in\bZ_+$,
we assume the set $\cZ=\cX\vcup\cY:=\{\ux_1,\uy_1,\ldots,\ux_\gq,\uy_\gq\}$ to be linearly ordered as follows: $\ux_1\succ\uy_1\succ\ldots\succ\ux_\gq\succ\uy_\gq$.

The set of all words in $\cZ$ is denoted by $\Word(\cZ)$; the blank or empty word is denoted by $\un$.
Any word $\omega$ will be indicated as $\omega=\uz_{i_1}\ldots\uz_{i_r}$ where each $\uz_{i_{k}}$ is one of the elements from $\cZ$.
For $\omega = \omega_1 \omega_2 \omega_3\in \Word(\cZ)$, we call $\omega_i$ a \emph{subwords} of $\omega$, $i=1,2,3$.
We say that the word $\omega$ has \emph{length} or \emph{multiplicative degree} $k$, written $\mdeg\omega = k$
if it belongs to $\cZ^k={\underbrace{\cZ\cdots\cZ}}_k$. By convention, $\mdeg(\un){=}0$.
The set of all words of multiplicative degree $k$ will be denoted $\Word_{k}(\cZ)$. We have $\Word_{0}(\cZ)=\{1\}$, $\Word_{\leqslant k}(\cZ)=\bigocup_{i=0}^k\Word_{i}(\cZ)$, and $\Word(\cZ)=\bigocup_{k{\geqslant}0}\Word_{k}(\cZ)$.

\smallskip~

\begin{definition}[of bimonomials]\label{def_bimonomial}
For a word $\ww\in\Word(\cZ)$, and for a variable $\uz\in\cZ$,
we define $\deg_{\uz}\!\ww$ to be the number of occurrences of $\uz$ in $\ww$,
and the total degree of a word $\ww$ as $\tdeg\ww=\sum_{i\in\Jq} |\deg_{\ux\is}\!\ww{-}\deg_{\uy\is}\!\ww|$.

A word $\ww$ at the generating set $\cZ=\cX\vcup\cY$ is called \emph{standard bimonomial}
if it is of the form $\ww=\uz_1^{k_1}\cdots \uz_\gq^{k_\gq}$, $k_i\geqslant{0}$ for any $i\in\Jsq$,
where $\uz_i$ equals either $\ux_i$ or $\uy_i$.
The total degree of the standard bimonomial $\ww=\uz_1^{k_1}\cdots \uz_\gq^{k_\gq}$ equals $\tdeg\ww=\sum_{i\in\Jq}k_i$.

The set of all standard bimonomials in $\cZ$ is denoted by $\Mon(\cZ)$;
let $\uii:\Mon(\cZ)\hookrightarrow\Word(\cZ)$ be the set injection.
We denote $\Mon_{d}(\cZ):=\Mon(\cZ)\cap\Word_{d}(\cZ)$.
For any word $\ww\in\Word(\cZ)$ there exists a uniquely defined standard bimonomial
\begin{align}\label{equ_standard_bimonomial}
\bww=\uz_1^{k_1}\cdots \uz_\gq^{k_\gq}\in\Mon(\cZ),\quad
\uz_i^{k_i}=\left\{
\begin{array}{ll}
\ux_i^{k_i}, & \hbox{if}\;\; k_i\geqslant{0}, \\
\uy_i^{-k_i}, & \hbox{if}\;\;  k_i<{0},
\end{array}
\right.\quad
k_i{=}\deg_{\ux\is}\!(\ww){-}\deg_{\uy\is}\!(\ww).
\end{align}
This defines a surjection
$$\varrho:\Word(\cZ)\twoheadrightarrow\Mon(\cZ),\quad \ww\mapsto\varrho(\ww)=\bww\,,$$
also, $\varrho:\Word_{\leqslant{d}}(\cZ)\twoheadrightarrow\Mon_{\leqslant{d}}(\cZ)$.
\end{definition}

We fix the set of generators $\cZ=\cX\vcup\cY:=\{\ux_1,\uy_1,\ldots,\ux_\gq,\uy_\gq\}$, given in ascending order.
\emph{The free associative monoid} with generating set $\cZ$ is defined as the set $\gM_{\!\cZ}=\Word(\cZ)$,
equipped with the multiplication $\gM_{\!\cZ}\times\gM_{\!\cZ}\longrightarrow\gM_{\!\cZ}$ obtained by the concatenation of words,
the identity equal to the blank word, and the involution $*:\gM_{\!\cZ}\to\gM_{\!\cZ}$ with
$(\ww^*)^*=\ww$, $(\ww\uu)^*=\ww^*\uu^*$, $\un^*=\un$ for any $\ww,\uu\in\gM$
such that $*$ sends $\ux\is\mapsto\uy\is$, $\uy\is\mapsto\ux\is$.

\begin{definition}\label{def_comm_monoid}
Let $\bgM_{\!\cZ}$ denote the quotient of $\gM_{\!\cZ}$ by smallest congruence containing $\ux_i\ux_j = \ux_j\ux_i$; $\uy_i\uy_j = \uy_j\uy_i$; $\ux_i\uy_i = \uy_i\ux_i = 1$ for any $i,j\in\Jsq$.
The $\bgM_{\!\cZ}$ can be identified with a set of standard bimonomials of a type \eqref{equ_standard_bimonomial},
uniquely expressed in ascending order of indices.
The projection map $\varrho:\gM_{\!\cZ}\to\bgM_{\!\cZ}$ defines the structure of associative commutative monoid $\bgM_{\!\cZ}$
with the multiplication given by:
$$\bgM_{\!\cZ}\times\bgM_{\!\cZ}\longrightarrow\bgM_{\!\cZ},\quad \ww_1\times\ww_2=\varrho(\ww_1\ww_2).$$
\end{definition}

Note there is a natural identification map $\uii: \bgM_{\!\cZ} \rightarrow \Mon(\cZ) \subset \Word(\cZ)$. Moreover, there is a canonical isomorphism $\psi: \bgM_{\!\cZ} \rightarrow \mathbb{Z}^{\otimes \gq}$ with $\ux_i \mapsto e_i,\, \uy_i \mapsto -e_i$,\, $i=1, \ldots, \gq$, where $e_1, \ldots e_\gq$ are the canonical basis of $\mathbb{Z}^{\otimes \gq}$; the involution acts as the additive inverse $-$.

Both monoids $\gM$ and $\bgM$ are involuted in the sense that they are equipped with an involution map $*$, acting according to rule
$*:\ux\is\leftrightarrow\uy\is$, and hence $\bgM\ni\ww \overset{(*)}{\to} \varrho(\ww^{\,*})\in\bgM$\,.

\vspace{0.1cm}

Further in the text, standard bimonomials will be called simply bimonomials, since we do not use others.

\vspace{0.2cm}

\section{Rational TGW algebra}

\subsection{Definition of RTGW algebra}

\smallskip~
\smallskip~

Now we present a construction of the rational twisted generalized Weyl (RTGW) algebra based on the definition of the TGW algebra given in \cite{FH}, \cite{MT-99}, \cite{MT-02}. We consider the following data:

\begin{definition}[RTGW datum]\label{def_RTGW_datum}
Let $\gp,\gq\in\bZ_+$, $0<\gq\leqslant{\gp}$.
Let $\Jsp={\le{1;\gp}\ra}$ be the index set,
and $\Jsp=\Jsq\cup\le \gq{+}1;\gp\ra=\ocup_{k=1}^{r{+}1} J_k$ be a partition of length $r+1$
provided $\Jsq={\le{1;\gq}\ra}=\ocup_{k=1}^r J_k$, $J_{r{+}1}=\le \gq{+}1;\gp\ra$.

\vspace{0.2cm}

Let $\LL=\gS\mo\!\La$ be a localization of the involuted polynomial ring $\La=\bk[\,\glambda\,]^{\gzeta}$
where $|\glambda|=|\gzeta|=\gp$, let $\Aut(\La)$ be an automorphism ring.

For the generator sets $\cZ=\cX\vcup\cY$, $|\cX|=|\cY|=\gq\,{\leqslant}\, \gp$,\, we consider the following data:
\begin{itemize}

\item
let $\Si=\Si\qsub=\langle \ssi\is\rangle_{i\in\Jq}\subseteq\Aut(\La)$ be the shift group of $\La$, $\Si\simeq\bZ^{\otimes{\gq}}$;

\item
let $\ft=(\tt_{1},\ldots,\tt_{\!\gq})$, $\tt\is\in\LLt$, be the $\gq$-tuple of the invertible elements from $\LLt$;

\item
let $*:\La\overset{\sim}{\longrightarrow}\La$ be an involution such that (a) $\la\is\mapsto\bbla\is=\zeta\is-\la\is$;\, (b) $\tt\is^{\!*}=\ssi\is(\tt\is)$;\, $*:\gS\to\gS$;

\item
if $\{\la_{i_1},\ldots,\la_{i_k}\}$ is a one part of the given partition, then we assume $\gzeta_{i_1}=\ldots=\gzeta_{i_k}$;

\item
let $\bS=\bS_{\!1}\times\bS_{2}\times \ldots \times \bS_{r{+}1}\subset\Aut_\bk(\La)$ be a direct product of symmetric groups corresponding to the partition acting on generators $\glambda$ and $\cZ$ simultaneously by the permuting of indexes belonging to the same partition;

\item
let $\bG=\bS\times\langle*\rangle\subset\Aut_\bk(\La)$, and $\Si\cap\bG=\id_{\La}$;

\item
let $\vmu=(\!(\muxx_{ij})\!)\subset\Mat_{\gq\times\gq}(\bL)$ be a $\gq\!\times\!\gq$-matrix with $\muxx_{ij}\in\LLt$, $\muxx_{ij}\muxx_{ji}=1$, and $\muxx_{ii}=1$;

\item
let there is an involution $*:\Mon(\cZ)\to\Mon(\cZ)$ such that $*:\ux\is\leftrightarrow\uy\is$;

\item
the involution $*$ satisfies : ${*}\ssi{*}=\ssi\mo\in\Si$
for any $\ssi\in\Si$ and we have
\begin{align}\label{equ_involution_rule}
\big(\ssi(\ff)\big)\st=\ssi\mo\!(\ff^*),\qquad \forall \;\; \ssi\in\Si,\;\;\; \forall\;\; \ff\in\LL\,.
\end{align}
\end{itemize}
\end{definition}

\vspace{0.1cm}

\begin{definition}\label{def_RTGWA_general}
Given RTGW datum, a rational twisted generalized Weyl (RTGW) algebra $\fA=\Alg_{\!\gq,\gp}\big(\gS,\bG,\Si,\ft,\vmu\big)$
is defined as an associative $\bk$-algebra generated by the elements of the sets $\glambda$, $\cZ$
with unital injective ring homomorphism $\rho:\La\to\fA$ such that $\rho:\gS\to\gS$ and hence $\rho:\LL\to\LL$,
modulo the following defining relations:
\begin{align}\label{equ_C-construction-relations_two}
\begin{array}{rl}
\ux\is\ff=\ssi\is(\ff)\,\ux\is, & \uy\is\ff=\ssi\is\mo(\ff)\,\uy\is,
\\
\uy\is\ux\is=\tt\is, & \ux\is\uy\is=\ssi_{i}(\tt\is),
\\
& \ux\is\ux\js=\muxx_{ij}\,\ux\js\ux\is,
\end{array}\qquad\qquad i,j\in\Jsq, \; \ff\in\LL,
\end{align}
provided that the following equations are satisfied for any $i,j,k\in\Jsq$, $k\not\in\{i,j\}$:

\begin{align}\label{equ_mu_conditions_one}
\begin{array}{ccccc}
(\muxx_{ij})^{\mo}=(\muxx_{ij})^{*}; && \ssi\is\ssi\js(\muxx_{ij})=\muxx_{ij}; &&\; \ssi\ks(\muxx_{ij})=\muxx_{ij};
\end{array}
\end{align}
\begin{align}\label{equ_mu_conditions_two}
\begin{array}{c}
\ssi_{\!i}(\muxx_{ij})\,\ssi_{\!j}(\muxx_{ij})=\dfrac{\ssi_j\mo(\bt_i)\,\bt_j}{\bt_i\,\ssi_i\mo(\bt_j)};\qquad
\ssi_{\!i}(\muxx_{ij})\,\ssi_{\!j}\mo\!(\muxx_{ij})=\dfrac{\ssi_j\mo(\bt_i)\,\ssi_i(\tt_j)}{\bt_i\,\tt_j};
\end{array}
\end{align}
\begin{align}\label{equ_mu_conditions_three}
\ssi\is(\tt\js)\ssi\ks^{\pm1}(\tt\js)=\tt\js\,\ssi\is\ssi\ks^{\pm1}(\tt\js),
\end{align}
where $\bt\is:=\tt\is^{\!*}=\ssi_{i}(\tt\is)$, $i\in\Jsq$.

We assume that the equalities obtained by the involution $*$ are also true as the original ones,
in particular, $\uy\is\uy\js = (\muxx_{ij})^* \,\uy\js\uy\is$ holds.
This principe will be called a rule of involution.
\end{definition}

Note that by the definition $\ux\is,\uy\is$ are invertible in $\fA$ from the left and from the right.

\vspace{0.2cm}

\begin{remark}\label{rem_dagger}
Applying the involution rule \eqref{equ_involution_rule}, from $\ux\is\ff = \ssi\is(\ff)\ux\is$ for $\ff\in\LL$ we get $\uy\is\ff^* = \left(\ssi\is(\ff)\right)^* \uy\is =\left(\ssi\is(\ff)\right)^* \uy\is =\ssi\is\mo(\ff^*)\uy\is$.
\end{remark}

\vspace{0.1cm}

\begin{remark}[reduction law]\label{rem_reduction_law}
Since a ring homomorphism $\rho:\La\to\fA$ is unital injective,
then any expression in algebra can be canceled from the left or right by an invertible element of the ring $\LL$, as well as by any generators belonging to the set $\cZ$.
\end{remark}

\vspace{0.2cm}

It is because any element from $\cZ$ is invertible in $\fA$ by \eqref{equ_C-construction-relations_two}.
In the following proofs, we will use this property without reference.

\vspace{0.2cm}

In the definition \ref{def_RTGWA_general} of RTGW algebra, the multiplication law is defined
using the matrix $\vmu=(\!(\muxx_{ij})\!)$ for any two elements from $\cX$ only.
But conditions \eqref{equ_mu_conditions_one} - \eqref{equ_mu_conditions_three}
allow one to determine the law of multiplication for all elements of the set $\cZ$, as the following lemma shows.

\begin{lemma}\label{lem_mu_equalities_RTGWA}
For any $i,j\in\Jsp$, we put
\begin{align}\label{lem_muxy_condition}
\muyy_{ij}=(\muxx_{ij})^*,\qquad
\muxy_{ij}=\ssi\is(\muxx_{ij})\,\dfrac{\bt\is}{\ssi\mo\js(\bt\is)}\overset{\eqref{equ_mu_conditions_two}}{=} \ssi\mo\js(\muxx_{ji})\dfrac{\ssi\is(\tt\js)}{\tt\js},\qquad \muyx_{ij}=(\muxy_{ij})^{*}\!.
\end{align}

Then RTGW algebra $\fA$ satisfies the following dependencies
\begin{align}\label{equ_mu_commutant_two_new}
\ux\is\ux\js=\muxx_{ij}\ux\js\ux\is,\quad
\uy\is\uy\js=\muyy_{ij}\uy\js\uy\is,\quad
\ux\is\uy\js=\muxy_{ij}\uy\js\ux\is,\quad
\uy\is\ux\js=\muyx_{ij}\ux\js\uy\is,\quad i,j\in\Jsq.
\end{align}

\vspace{0.2cm}

Moreover, for any $i,j,k\in\Jq$, $k\not\in\{i,j\}$ we obtain:
\begin{align}\label{equ_C-construction-relations_two_xy}
\begin{array}{ccccc}
\muxy_{ij}=\muxy_{ji};
&& \ssi\is(\muxy_{ij})=\ssi\js(\muxy_{ij}); &&\; \ssi\ks(\muxy_{ij})=\muxy_{ij}.
\end{array}
\end{align}
\end{lemma}

\begin{proof}
First, using the principle of involution, we obtain $\uy\is\uy\js=(\ux\is\ux\js)^* \overset{\eqref{equ_C-construction-relations_two}}{=}(\muxx_{ij}\ux\js\ux\is)^*=\muyy_{ij}\uy\js\uy\is$.

To prove $\ux\is\uy\js=\muxy_{ij}\uy\js\ux\is$ we calculate\,
$\ux\is\uy\js\cdot\uy\is = \ssi\is(\muyy_{ji})\,\bt\is\uy\js= \ssi\is(\muxx_{ij})\,\bt\is\uy\js=\muxy_{ij}\ssi\mo\js(\bt\is)\uy\js =\muxy_{ij}\uy\js\ux\is\cdot\uy\is$ whence $\ux\is\uy\js{=}\muxy_{ij}\uy\js\ux\is$ follows.

Finally, we have $\uy\is\ux\js=(\ux\is\uy\js)^* {=}(\muxy_{ij}\uy\js\ux\is)^*=\muyx_{ij}\ux\js\uy\is$,
which completes the proof of the equalities \eqref{equ_mu_commutant_two_new}.

The first equality of \eqref{equ_C-construction-relations_two_xy} follows from
$\muxy_{ij}=\ssi\is(\muxx_{ij})\,\dfrac{\bt\is}{\ssi\mo\js(\bt\is)} \overset{\eqref{equ_mu_conditions_two}}{=}\ssi\js(\muxx_{ji})\,\dfrac{\bt\js}{\ssi\mo\is(\bt\js)} =\muxy_{ji}$.

Further, we get the second equality of \eqref{equ_C-construction-relations_two_xy} since
$$\ssi\js\ssi\is\mo(\muxy_{ij})= \ssi\js\ssi\is\mo\!\left(\!\ssi\is(\muxx_{ij})\,\dfrac{\bt\is}{\ssi\mo\js(\bt\is)}\right)
=\ssi\js(\muxx_{ij})\,\dfrac{\ssi\js(\tt\is)}{\tt\is} \overset{\eqref{equ_mu_conditions_one}}{=}
\ssi\mo\is\!(\muxx_{ij})\dfrac{\ssi\js(\tt\is)}{\tt\is}\overset{\eqref{lem_muxy_condition}}{=}\muxy_{ji}=\muxy_{ij}.$$

Finally, we get
$\ssi\ks(\muxy_{ij}){=}\ssi\ks\!\left(\!\ssi\mo\js\!(\muxx_{ji}\!)\dfrac{\ssi\is(\tt\js)}{\tt\js}\right) \overset{\eqref{equ_mu_conditions_two}}{=}
\ssi\mo\js\!(\muxx_{ji})\dfrac{\ssi\ks\ssi\is(\tt\js)}{\ssi\ks(\tt\js)} \overset{\eqref{equ_mu_conditions_three}}{=} \ssi\mo\js\!(\muxx_{ji})\dfrac{\ssi\is(\tt\js)}{\tt\js}{=}\muxy_{ij}$.
\end{proof}


\begin{corollary}
The values of the structure constants $\muxy_{ij},\muyx_{ij},\muyy_{ij}$ for any indices $i,j$
are uniquely determined by the values of $\muxx_{ij}$, $\tt\is$ and their shifts.
For any $i,j\in\Jsq$, $\tt_i,\bt_i,\muxx_{ij},\muxy_{ij},\muyx_{ij}, \muyy_{ij}\in\LLt$.
These parameters obey the relations:
$\muxy_{ii}={\bt_i}/{\tt_i}$,\, and\;
$\muxx_{ij}\muxx_{ji}=\muyy_{ij}\muyy_{ji}=\muxy_{ij}\muyx_{ji}=1$.
\end{corollary}

\vspace{0.1cm}

\begin{corollary}\label{cor_mu_tt_connections}
In algebra $\fA$, there are such dependencies between the entered parameters:
\begin{gather}\label{equ_mu_tt_connections}
(a)\;\; \muyx_{ij}{=}\dfrac{\tt_i}{\ssi_j(\tt_i)}\ssi\mo_i(\muxx_{ji});\quad
(b)\;\; \muyy_{ij}{=}\dfrac{\tt_i\ssi\mo_i(\tt_j\!)}{\ssi\mo_j(\tt_i\!)\tt_j}\muxx_{ij}.
\end{gather}
\end{corollary}

\begin{proof}
For the proof of two first equalities, using the assumption of associativity, we conclude

$(a)\; \muyx_{ij}\ssi_j(\tt_i)\ux_j = \uy_i\ux_j\ux_i = \ssi\mo_i(\muxx_{ji})\,\tt_i\ux_j$;

$(b)\quad \muyy_{ij}\ssi\mo_j(\tt_i\!)\tt_j=\uy_i\uy_j\ux_i\ux_j =\ssi\mo_i\ssi\mo_j(\muxx_{ij})\ssi\mo_i(\tt_j\!)\tt_i
\,\overset{\eqref{equ_mu_conditions_one}}{=}\, \muxx_{ij}\ssi\mo_i(\tt_j\!)\tt_i$.\quad
Q.E.D.
\end{proof}

\vspace{0.4cm}

We can consider the RTGW algebra as $\bZ^{\!\gq}$-graded by putting $\deg\ux \is = e_i$, $\deg\uy \is = - e_i, i=1,\ldots,\gq$,
where $\{ e_i \}_{i=1}^\gq$ is the canonical basis of $\bZ^{\!\gq}$.

\vspace{0.6cm}

\begin{proposition}\label{prop_RTGWA_crossprod}
The algebra $\fA$ is a crossed product ring of $\LL$ and $\bZ^{\!\gq}$. It is an Ore domain.
\begin{proof}
Each homogeneous component of $\fA$ with the above graduation contains an invertible element. In fact, by equation (2.1), since each $\tt \is \in \LLt$, each bimonomial on the $\ux \is, \uy \js$ is also invertible. Hence the statement about cross products follows from \cite[1.4]{NO1}.
Since the group $\bZ^\gq$ is an ordered group by \cite[13.1.6]{Pass1} and $\LL$ is a domain, the crossed product is a domain, \cite[Proposition 8.3]{Zhang}.
\end{proof}
\end{proposition}

\begin{corollary}\label{cor_elements_of_A}
Let $\fA=\Alg_{\!\gq,\gp}\big(\gS,\bG,\Si,\ft,\vmu\big)$ be RTGW algebra.
Consider the generating set $\cZ=\cX\vcup\cY$, and $\bgM=\bgM_{\!\cZ}$ as above defined.
Any element $\gF$ of $\fA$ can be reduced as finite sum
$$\gF={\osum}_{\bww\in\bgM}\ff_{\bww}\bww\in\LL\!\bgM$$
where $\LL\bgM={\obigoplus}_{\bww\in\bgM}\LL\bww$ denotes the left $\LL$-module with the countable base $\bgM$.
Here the sum is taken over different bimonomials, therefore the equality ${\osum}_{\bww\in\bgM}\ff_{\bww}\bww=0$
induces $\ff_{\bww}{=}0$ for any $\ww\in\gM$.
\end{corollary}

It follows from the skew commutativity property in $\fA$.

\vspace{0.2cm}

\begin{corollary}
For any $i,j$, the parameter $\muxx_{ij}$ depends only on the variables $\la\is$ and $\la\js$ (since $\ssi\ks(\muxx_{ij})=\muxx_{ij}$ by definition).
Moreover, $\muxx_{ij}$ is a function of $\la\is+\bbla\js$ over $\bk$ since $\ssi\is\ssi\js(\muxx_{ij})=\muxx_{ij}$.
Considering that $\muxx_{ji}=(\muxx_{ij})^{\pi_{ij}}$, we can assert that $\muxx_{ij}$ is a product of the fractions of a type
$\dfrac{\la\is{+}\bbla\js{+}c}{\bbla\is{+}\la\js{+}c}$, $c\in\bZ$.
Hence, $\muxy_{ij}$ is a product of the fractions of a type $\dfrac{\la\is{+}\la\js{+}c}{\bbla\is{+}\bbla\js{+}c}$.
\end{corollary}
\vspace{0.1cm}

\vspace{0.1cm}

\begin{proposition}\label{prop_FH_example}
The TGW algebra $\mathcal{A}$ (see \cite{FH}, definitions 2.2 and 2.3) can be viewed as an RTGW algebra $\fA$ over some localization of the basic ring $R$ for which the fulfillment of condition $\mu_{ij}\in\bk$ is not required.
\end{proposition}

\begin{proof}
First of all we show that for $i,j,k\in\Jsq$ pairwise different, and any $\ssi\in\Si$, the following equalities are true:
\begin{align}\label{equ_identities_TGWA}
\begin{array}{c}
\ssi\is\ssi_{\!j}(\tt_{i}\tt_{\!j})=\ssi_{\!j}(\muxy_{i\!j})\ssi_{i}(\muxy_{ji})\,\bt_{i}\bt_{\!j};
\\
\dfrac{\ssi\is(\tt\js)\ssi\ks(\tt\js)}{\tt\js\ssi\is\ssi\ks(\tt\js)}
=\dfrac{\ssi_{\!j}\mo(\muxx_{ij})\,\muxy_{ij}} {\ssi_{\!k}\!\left(\ssi_{\!j}\mo(\muxx_{ij})\,\muxy_{ij}\right)}
\,\text{and}\,\,
\dfrac{\ssi\is(\bt\js)\ssi\ks(\bt\js)}{\bt\js\ssi\is\ssi\ks(\bt\js)}
=\dfrac{\muxx_{ij}\ssi_{\!j}(\muxy_{ij})} {\ssi_{\!k}\!\left(\muxx_{ij}\ssi_{\!j}(\muxy_{ij})\right)}.
\end{array}
\end{align}

To proof \eqref{equ_identities_TGWA}, by \eqref{equ_mu_conditions_one} and \eqref{lem_muxy_condition}, we get:\;
$\ssi_{\!j}(\muxy_{i\!j})\cdot\ssi_{i}(\muxy_{ji})= \dfrac{\ssi_j(\bt_i)}{\bt_i}\muxx_{ij}\cdot\dfrac{\ssi_i(\bt_j)}{\bt_j}\muxx_{ji}=
\dfrac{\ssi\is\ssi_{\!j}(\tt_{i}\tt_{\!j})}{\bt_{i}\bt_{\!j}}$,
and
$\dfrac{\ssi\is(\tt\js)\ssi\ks(\tt\js)}{\tt\js\,\ssi\is\ssi\ks(\tt\js)}
=\dfrac{\ssi\is(\tt\js)}{\tt\js} : \ssi_{\!k}\!\left( \dfrac{\ssi\is(\tt\js)}{\tt\js}\right)
=\dfrac{\muxy_{ij}}{\ssi_{\!j}\mo(\muxx_{ji})}\cdot{\ssi_{\!k}\!\left(\dfrac{\ssi_{\!j}\mo(\muxx_{ji})}{\muxy_{ij}}\right)}
=\dfrac{\ssi_{\!j}\mo(\muxx_{ij})\,\muxy_{ij}} {\ssi_{\!k}\!\left(\ssi_{\!j}\mo(\muxx_{ij})\,\muxy_{ij}\right)}$.
Last equality is obtained from the previous, using the involution $*$, or directly.

\vspace{0.1cm}

We consider a TGW algebra $\mathcal{A}=\mathcal{A}_{\mu}(R,\sigma,t)$ defined in \cite{FH}, definitions 2.2 and 2.3,
where $R$ is an unital associative $\Bbbk$-algebra,
$\mu=(\mu_{ij})_{i{\ne}j}$ is a parameter $n{\times}n$ matrix without diagonal with $\mu_{ij}\in\Bbbk^{\times}$,
$\sigma: \bZ^n \to\Aut_{\Bbbk}(R)$ is a group homomorphism,
$\Sigma=\langle\sigma_1,\ldots,\sigma_n\rangle$ is an abelian group,
$t$ is a function $t:\{1,\ldots,n\}\to Z(R)$.

Assume that the ring $R$ contains all non-constant indecomposable dividers of the elements $\ssi(t_i)$ and their inverses for any $1\leqslant{i}\leqslant{n}$, $\sigma\in\Sigma$. Possibly, $R$ is some extension of the polynomial ring.

We put $\muxy_{ij}=\mu_{ij}\in\Bbbk^{\times}$, $\muxy_{ii}=\sigma_i(t_i)/t_i\in{R}$.
Since by the assumption of \cite{FH}, the group $\Si$ acts on $\mu_{ij}$ identically,
then the relations \eqref{equ_identities_TGWA} for the RTGW algebra take the form
(1.2) and (1.3) from \cite{FH} for the TGW algebra.

In TGW algebra $\mathcal{A}$, we define the parameters $\muxx_{ij}$ by the formulas
$\muxx_{ij} =\dfrac{1}{\muxy_{ij}}\cdot\dfrac{\ssi_i\ssi_j(\tt_j)}{\ssi_j(\tt_j)}$,
these parameters do not have to belong to the field $\bk$.
Using Example 2.8, \cite{FH}, we obtain equality $X\is X\js=\muxx_{ij} X\js X\is$,
this is consistent with Lemma \ref{lem_mu_equalities_RTGWA}.
\end{proof}

Note that of the examples given in the third section, the unitary algebra is actually a TGW algebra, while the orthogonal algebra is not.

\smallskip~

\subsection{Some properties of RTGW algebra}

\smallskip~
\smallskip~

Let $\fA=\Alg_{\!\gq,\gp}\big(\gS,\bG,\Si,\ft,\vmu\big)$, $1<\gq\leq{\gp}$, be RTGW algebra corresponding the storey partition
$\Jsp=\Jsq\cup\,{\le{\gq{+}1,\gp}\ra}= {{J}_{\!1}} \cup\ldots\cup\,{{J}_{r{+}1}}$,  ${J}_{r{+}1}={\le{\gq{+}1,\gp}\ra}$,
see \eqref{equ_storey_partition} in definition \ref{def_storey_partition}.

\smallskip~

\begin{proposition}\label{prp_center_RTGW}
Let $\fA$ be RTGW algebra.
If $\gq=\gp$ then the center of $\fA$ equals the base field $\bk$.
For the case $\gq<\gp$, denote by $\bar{\La}=\bk[\la_{\gq{+}1},\ldots,\la_{\gp}]\subset\La$ the subring,
and by $\bar{\gS}=\gS\cap\bar{\La}$ the multiplicative set.
Then the localized ring $\bar{\LL\!}=(\bar{\gS})\mo\bar{\La}$ coincides with the center of $\fA$.
\end{proposition}

\begin{proof}
Clearly, the localized ring $\bar{\LL\!}$ belongs to the center of the algebra $\fA$,
because the variables $\la_{\gq{+}1},\ldots,\la_{\gp}$ commute with all elements of $\fA$.

Suppose that some element $\gF$ from $\fA$, represented as $\gF={\osum}_{\bww\in\bgM}\ff_{\bww}\bww\in\LL\!\bgM$, belongs to the center of $\fA$.
Then, in particular, any summand $\ff_{\bww}\bww$ must commute with each generator $\la\is$.
If $\bww$ has a multiplier $\ux\is^k$, $k>0$, then $\ff_{\bww}\bww\cdot\la\is=(\la\is-k)\ff_{\bww}\bww$, which is impossible.
For the case $\bww=1$ and $\ff_{\bww}\not\in\bar{\LL\!}$ we obtain from the required condition $\ff_{\bww}\ux\is=\ux\is\ff_{\bww}$ that $\ff_{\bww}=\ssi\is(\ff_{\bww})$ for any $i=1,\ldots,\gq$, but this is not possible for a rational element from $\LL{\setminus}\bar{\LL\!}$
depending on $i$.
\end{proof}

There is a chain of RTGW subalgebras
\begin{align}\label{equ_flag}
\fA_{\vec{J}_{1}} \subset \fA_{\vec{J}_{2}} \subset \ldots \subset \fA_{\vec{J}_{r}} \subset \fA_{\vec{J}_{r{+}1}} =\fA,
\end{align}
where for each $t=1,\ldots,r{+}1$, subalgebra $\fA_{\vec{J}_{t}}$ is RTGW algebra with respect to the sub partition
$$\vec{J}_{t}={{J}_{\!1}} \cup\ldots\cup\,{{J}_{t}}
={\le{1;\gq_1}\ra}\cup\, {\le{\gq_1{+}1;\gq_2}\ra} \cup\ldots\cup\,{\le{\gq_{t{-}1}{+}1,\gq_t}\ra},$$
generated by the elements $\ux\is,\uy\is$, $i\in\vec{J}_{t{-}1}$ and $\la\is$, $i\in\vec{J}_{t}$ (we assume $\vec{J}_{0}=\varnothing$).
In particular, $\fA_{\vec{J}_{1}}=\LL_1\subset\LL$ is a localization of $\La_1=\bk[\la_1,\ldots,\la_{\gq_1}]$ corresponding to the first part of partition.

By Proposition \ref{prp_center_RTGW}, the center of algebra $\fA_{\vec{J}_{t{+}1}}$ equals $\LL_{t{+}1}$ that is a localization of polynomial ring $\La_{t{+}1}=\bk[\la_{\gq_t{+}1},\ldots,\la_{\gq_{t{+}1}}]$ with the multiplicative set $\gS_{t{+}1}=\gS\cap\La_{t{+}1}$.

\vspace{0.2cm}

For any $t=1,\ldots,r{+}1$ denote by $\gK_{\!\vec{J}_{t}}$ center of the algebra $\fA_{\vec{J}_{t}}$.
We denote by $\gK$ the subalgebra of $\fA$ generated by the centers $\gK_{\!\vec{J}_{t}}$ of all subalgebras $\fA_{\vec{J}_{t}}$,
$t=1,\dots,r{+}1$, it can be viewed as a rational analogue of the Gelfand-Zeitlin subalgebra of $\fA$.
For the case under consideration, the subalgebra $\gK$ coincides with the localized ring $\LL$,
it is a maximal commutative subalgebra of $\fA$.
Indeed, for any reduced element $\gF={\osum}_{\bww\in\bgM}\ff_{\bww}\bww\in\LL\!\bgM\setminus\LL$, and any $i\in\Jsp$ we have
$$[\la\is,\gF]={\osum}_{\bww\in\bgM}\ff_{\bww}[\la\is,\bww]={\osum}_{\bww\in\bgM}\ff_{\bww}(\la\is-\ssi_{\bww}(\la\is))\bww.$$
This sum equals zero only if $\la\is-\ssi_{\bww}(\la\is)=0$ which means $\bww$ does not depend on either $\ux\is$ or $\uy\is$ for any $i\in\Jsq$.
But this means that $\gF\in\LL$, which contradicts assumption.

The problem of determining the Gelfand-Zeitlin subalgebra and the maximal commutative subalgebra in the case
of an invariant subalgebra $\fA^{\!\bG}\subset\fA$ is somewhat more difficult to solve.

\subsubsection{The invariant subring $\La^{\!\!\bG}\subset\La$}

\smallskip~
\smallskip~

Now we describe the invariant subring $\La^{\!\!\bG}\subset\La$ of the ring $\La$ with respect to the action of the group $\bG=\bS\times\langle*\rangle$.
Follow the fundamental Theorem on symmetric polynomials, any polynomial of the invariant subring $\La^{\!\bS}$ is representable in a unique way
as a polynomial over $\bk$ in the elementary symmetric polynomials for all parts of the partition separately.

First of all note that $(\la\is-\ot\zeta\is)^*=-(\la\is-\ot\zeta\is)$ for any $i\in\Jsp$.
For the part $\J_{t}$ of the partition of $\Jsp$, let $\glambda_t=\{\la\is\}_{i\in\J_{t}}$.
Then $\bk[\glambda_t]^{\bS}$ is generated by symmetric polynomials in variables $\glambda_t$,
as well as symmetric polynomials in variables $\la\is-\ot\zeta\is$, $i\in\J_{t}$.
Denote by $\ue_\alpha(\glambda_t)$ the elementary symmetric polynomials in the variables $\{\la\is-\ot\zeta\is\}_{i\in\J_{t}}$, $\alpha=1,\ldots,|\J_{t}|$.
For any $\pi\in\bS$ and any polynomial $\ff$ in $\ue_\alpha(\glambda_t)$, $\alpha=1,\ldots,|\J_t|$, either or $\ff^{\pi}=\ff$ or $\ff^{\pi}=-\ff$.

The following statement is obvious, it describes the situation in the case when $*$ acts on the field $\bk$ identically.

\begin{remark}
Let $*|_\bk=\id_\bk$.
Any polynomial of the invariant subring $\La^{\!\bG}$ is a $\bk$-linear combination of products of elementary symmetric polynomials $\ue_\alpha(t)$ for all parts $\J_t$ of a given partition of $\Jsp$, and each such product contains an even number of factors satisfying $\alpha\equiv{1}\!\!\mod\!\!(2)$.
\end{remark}

\vspace{0.1cm}

For the exceptional case $\bk=\bC$, $\imath^*=-\overline{\imath}$, the product $\imath(\la\is-\ot\zeta\is)$ is an invariant of the involution $*$.
Let $\hat{\ue}_\alpha(\glambda_t)=\imath^{\alpha}\,\ue_\alpha(\glambda_t)$ denotes the $\alpha$-th elementary symmetric polynomial in variables $\imath(\la\is-\ot\zeta\is)$, $i\in\J_{t}$.
It is easy to check the following statement

\begin{remark}
Let $\bk=\bC$, $\imath^*=-\overline{\imath}$. The invariant subring $\La^{\!\bG}\subset\La$ is generated over $\bk$
by the elementary symmetric polynomials $\hat{\ue}_\alpha(t)$ for all parts $\J_t$ of a given partition of $\Jsp$.
\end{remark}

\vspace{0.1cm}

\begin{remark}
Denote $\LL^{\!\!\bG}\subset\LL$ the invariant subring, $\La^{\!\!\bG}\subset\LL^{\!\!\bG}$.
By the definition \ref{def_diskr_def}, the defining polynomial $\ttd$ is multi symmetric with respect to the given partition of $\Jp$,
and $\ttd^*=\ttd$, hence $\ttd\in\La^{\!\bG}$, moreover $\ttd^{\ssi}\cdot\ttd^{\ssi\mo} \in\La^{\!\bG}$ for any $\ssi\in\Si$.

It is easy to show that any element from $\LL^{\!\!\bG}$ can be represented as a sum of finite products of a type:\,\,\,
$F=\ff\oprod_{\ssi\in\Si}\dfrac{1}{\ttd^{\ssi}\,\ttd^{\ssi\mo}},\,\, \text{where}\,\, \ff\in\La^{\!\!\bG}$.
(We assume $\ff$ is not divisible by $\ttd^{\ssi}\ttd^{\ssi\mo}$ without remainder for any $\ssi$,
although this fraction can be reduced.)
\end{remark}

\vspace{0.1cm}

\subsubsection{The maximal commutative subalgebra of $\fA^{\!\bG}$}

\smallskip~
\smallskip~

Let $\fA=\Alg_{\!\gq,\gp}\big(\gS,\bG,\Si,\ft,\vmu\big)$, $1<\gq\leqslant{\gp}$, be RTGW algebra
corresponding to the partition
$\Jsp=\Jsq\cup\,{J}_{r{+}1}= {{J}_{\!1}} \cup\ldots\cup\,{{J}_{r{+}1}}$,  ${J}_{r{+}1}={\le{\gq{+}1,\gp}\ra}$.
Denote by $\bar{\La}=\bk[\la_{\gq{+}1},\ldots,\la_{\gp}]\subset\La$ the subring, and by $\bar{\gS}=\gS\cap\bar{\La}$ the multiplicative set.
As shown above, the localized ring $\bar{\LL\!}=\bar{\gS}\mo\bar{\La}$ is the center of the algebra $\fA$.

\begin{proposition}\label{prp_invariant_subalgebra}
Let $\fA^{\!\bG}\subset\fA$ be the invariant subalgebra.
The following conditions are satisfied.

\begin{enumerate}
\item\label{it_invariant_subalgebra_a}
If $\mathfrak{B}\subset\fA^{\!\bG}$ is a subalgebra under the condition $\La^{\!\!\bG}\subset\mathfrak{B}\nsubseteq \LL^{\!\!\bG}$,
then $\mathfrak{B}$ is non commutative.

\item\label{it_invariant_subalgebra_b}
The invariant subalgebra $\LL^{\!\!\bG}$ is a maximal commutative subalgebra of $\fA^{\!\bG}$.
\end{enumerate}
\end{proposition}

\begin{proof}
By corollary \eqref{cor_elements_of_A}, any element $\gF\in\fA$ takes the form $\gF={\osum}_{\bww\in\bgM}\ff_{\bww}\bww\in\LL\!\bgM$.
Let $\gF\not\in\fA\setminus\LL$, then $\ell\in\LL$ commutes with $\gF$ if and only if any monomial included in the sum with a nonzero coefficient commutes with $\ell$.

Given  monomial $\ww=\uz_1^{k_1}\cdots\uz_\gq^{k_\gq}\in\bgM$ with $\uz_i\in\{\ux_i,\uy_i\}$ we denote $\bar{k}_i=k_i$ if $\uz_i=\ux_i$ and $\bar{k}_i=-k_i$ if $\uz_i=\uy_i$.
Then for any $\ff\in\bk[\la_1,\ldots,\la_r]$ we get $\ww\ff=\ssi_{\ww}(\ff)\ww$ with $\ssi_{\ww}=\ssi_{1}^{\bar{k}_1}\cdots\ssi_{r}^{\bar{k}_r}$.

For the part $\J_{t}$ of the partition of $\Jsp$, denote by $\gamma_{\alpha}(\glambda_t)$ the $\alpha$-th elementary symmetric polynomials in the variables $\{\la\is\}_{i\in\J_{t}}$, $\alpha=1,\ldots,|\J_{t}|$.
Besides, we denote $\bar{\ue}_\alpha(\glambda_t)=\hat{\ue}_\alpha(\glambda_t)$ for the case $\bk=\bC$ and $\imath^*=-\imath$ and $\bar{\ue}_\alpha(\glambda_t)={\ue}_\alpha(\glambda_t)$ otherwise.
Evidently, any monomial $\ww\in\bgM$ commutes with polynomials $\bar{\ue}_1(\glambda_t)$ and $\bar{\ue}_2(\glambda_t)$ iff it commutes
with $\ga_{1}(\glambda_t)$ and $\ga_{2}(\glambda_t)$.

Let $\ww=\uz_1^{k_1}\cdots\uz_\gq^{k_\gq}\in\Mon(\cZ)$.
For $t$-th part of the given partition,
we assume $\J_t=\{i,\ldots,j\}$, and denote $\bar{k}=\bar{k}_i+\cdots+\bar{k}_j$, $\ga_{r}=\ga_{r}(\glambda_t)$.
We get
\begin{align*}
\begin{array}{crlcl}
{[}\,\ga_{1},\,\ww\,{]} &{\!=\!}& \big(\ga_{1}{-}\ssi_{i}^{\bar{k}_{i}}\cdots\ssi_{j}^{\bar{k}_{j}}(\ga_{1}) \big)\ww &{\!=\!}& (\bar{k}_i+\cdots+\bar{k}_j)\ww=\bar{k}\ww,
\\
{[}\,\ga_2,\,\ww\,{]} &{\!=\!}& \big(\ga_2{-}\ssi_{i}^{\bar{k}_i}\cdots\ssi_{j}^{\bar{k}_j}(\ga_2)\big)\,\ww
&{\!=\!}& \big((\bar{k}-\bar{k}_i)\la_i +\ldots +(\bar{k}-\bar{k}_j)\la_j\big)\ww{-}\gamma_2(\bar{k}_i,\ldots,\bar{k}_j) \ww.
\end{array}
\end{align*}
This implies, if ${[}\ga_{1}\!,\ww\,{]}={[}\ga_2,\ww\,{]}=0$, then $k_i=\ldots=k_j=0$,
because $\bar{k}=0$ and $\bar{k}=\bar{k}_i=\ldots=\bar{k}_j$.

Since this reasoning can be repeated for any part of the partition, excluding the last one, then as a result we obtain:\, an arbitrary monomial $\ww\in\Mon(\cZ)$ can commute only with the elements from subalgebra $\bar{\La}^{\!\bG}=\bk[\la_{\gq{+}1},\ldots,\la_{\gp}]^{\!\bG}\subset\La$.
Therefore $\mathfrak{B}\subseteq \LL^{\!\!\bG}$ in the contradiction with assumption.

The first statement of Proposition is done, and the second follows obviously from the first.
\end{proof}

\begin{remark}\label{rem_invariant_subalgebra}
Note that the subalgebras $\bar{\La}^{\!\bG} =\bar{\La}\cap\La^{\!\bG}$ and $\Quot(\bar{\La}^{\!\bG})\cap\LL^{\!\!\bG}$
belong to the center of $\fA^{\!\bG}$ since the variables $\la_{\gp{+}1},\ldots,\la_\gq$ commute with any element of $\bgM$.
\end{remark}

\vspace{0.1cm}

\section{Implementation of RTGW algebras on specific examples}

\vspace{0.2cm}

In this section we will prove that the class of RTGW algebras is big enough to contain the enveloping algebras of unitary ans orthogonal algebras of small rank.

\subsection{The enveloping algebras of special linear and unitary Lie algebras}

\smallskip~
\smallskip~

Let $\op{gl}_n(\bC)$ be simple finite-dimensional complex Lie algebra of $n\times n$ matrices,
$\op{U}(\op{gl}_n(\bC))$  be its universal enveloping algebra and $\Gamma(\op{gl}_n(\bC))$ its Gelfand-Zetlin subalgebra (\cite{DOF}, \cite{Maz-ort}).
We will denote by $E_{ij}$ standard matrix units.

\subsubsection{The standard matrix generators of algebras $\sl(3;\bC)$ and $\su(3;\bC)$}

\smallskip~
\smallskip~

Let $\Mat_{3\times 3}(\bC)$ be a space of matrix with basis given by the matrix units $E_{kl}\in M_{3\times 3}(\bC)$,$1\leqslant k,l\leqslant 3$,
whose $(k,l)$-entry is one, and all other entries are zeros.
The elements $E_{kl}$, $k\ne l$, and $E_{11}-E_{22},E_{22}-E_{33}$ are considered as generators
of the matrix special linear Lie algebra $\sl(3,\bC)=\left\{M\in \operatorname{gl}(3;\bC):\operatorname {tr} (M)=0\right\}$ and of its universal enveloping algebra $\op{U}(\sl(3,\bC))$.

\smallskip~

In order to define the generators of the complex unitary Lie algebra $\su(3;\bC)$, we use the concept of so-called Gell-Mann matrices
$F_i\in\Mat_{3\times 3}(\bC)$, $i=1,\ldots,8$ \cite{TU}, developed by Murray Gell-Mann, there is a set of eight linearly independent
$3\times 3$ traceless Hermitian matrices that span the Lie algebra $\su(3;\bC)$ of the $\operatorname{SU}(3)$ group
and the universal enveloping algebra $\op{U}(\su(3,\bC))$.

\begin{lemma}\label{lem_Gell_Mann_matrix}
Gell-Mann matrices belonging to $\Mat_{3\times 3}(\bC)$ are given by the presentation
\begin{gather}\label{equ_Gell_Mann_matrix}
\begin{array}{llll}
F_1{=}E_{21}{+}E_{12}, & F_2{=}\imath(E_{21}{-}E_{12}), & F_3{=}E_{11}{-}E_{22}, & F_4{=}E_{31}{+}E_{13},
\\
F_5{=}\imath(E_{31}{-}E_{13}), & F_6{=}E_{32}{+}E_{23}, & F_7{=}\imath(E_{32}{-}E_{23}),
& F_8{=}\dfrac{1}{\sqrt{3}}(E_{11}{+}E_{22}{-}2E_{33}).
\end{array}
\end{gather}
Elements $\{2\mo F_i \mid 1\leqslant i\leqslant 8\}$ satisfy anti-commutation relations
$$\begin{aligned}\left[F_{a}, F_{b}\right]&=2\imath{\sum} _{c}\,f^{abc}F_{c} \end{aligned}\quad
\text{with the structure constants}\quad
\begin{aligned}f^{abc}&=-\imath{\frc{}{\!4}}\operatorname{tr} (F_{a}[F_{b},F_{c}]),\end{aligned}$$
they are the generators of complex unitary Lie algebra $\su(3;\bC)$.
\end{lemma}

\begin{remark}
Along with the standard involution of an algebra $\gl(n,\bC)$, which is a complex conjugate transpose,
the algebra $\su(3;\bC)$ is also endowed with an involution defined by equalities
$$F\is^*=F\is,\,\, i=1,2,6,7,\qquad F\is^*=-F\is,\,\, i=3,4,5,8,$$
which can be easily verified directly.
Therefore, $(\imath{F\is})^*=\imath{F\is}$ for $i=3,4,5,8$.
\end{remark}


For the examples below, we explicitly indicate the value of the parameters from definition \ref{def_RTGWA_general} of RTGW algebra
satisfying the conditions \eqref{equ_mu_commutant_two_new}, \eqref{equ_C-construction-relations_two_xy}.
The proposed construction and its justification can be stated for an arbitrary rank,
but, in order to avoid extensive calculations, later in this section we consider in detail only the cases of rank three.


\subsubsection{Structure datum}

\smallskip~
\smallskip~

Let $\gq=3$, $\gp=6$. We consider the RTGW datum $\big(\gS,\bG,\vsigma,\ft,\vmu\big)$ (see definition \ref{def_RTGW_datum}) that have the following specific features:

\mat
we denote by $\Jsq=\{\small\text{{11,\,21,\,22}}\}$ and $\Jsp=\{\small\text{{11,\,21,\,22,\,31,\,32,\,33}}\}$ the sets of double indices given in ascending order under the storey partition
$\Jsp=J_1\cup J_2\cup J_3=\{\small\text{{11}}\} \cup \{\small\text{{21,\,22}}\}\cup\{\small\text{{31,\,32,\,33}}\}$.

\mat
let $\La=\bC[\,\glambda\,]^{\gzeta}$ be an involuted polynomial ring with
$\glambda=\{\la_\tki\}_{\tki\in\Jp}$, $\gzeta=\{\gzeta_\tki\}_{\tki\in\Jp}$
where $\gzeta_{11}{=}0$, $\gzeta_{21}{=}\gzeta_{22}{=}1$, $\gzeta_{3j}{=}0$, $j=1,2,3$;
an involution $*\in\Aut_{\bR}(\La)$ acts on $\La$ by the rule: $\la_\tki^*=\bbla_\tki=\zeta_\tki-\la_\tki$, besides,  $(a{+}\imath{b})^*=a{-}\imath{b}$, $a,b\in\bR$;

\mat
let $\Si\in\Aut_{\bC}(\La)$ be an abelian shift group, $\Si=\langle\ssi_\tki\rangle_{\tki\in\Jp} \simeq \bZ^{\otimes{6}}$
acting on $\La$ as:\, $\ssi_\tki(\la_\tki)=\la_\tki-1$,\, $\ssi_\tki(\la_\tmj)=\la_\tmj$, $\tmj\ne\tki$,
in addition, $(\ssi(\ff))^*=\ssi\mo(\ff^*)$, $\ssi\in\Si$, $\ff\in\La$;

\mat
let $\cZ=\{\ux_{11},\uy_{11},\ux_{21},\uy_{21},\ux_{22},\uy_{22}\}$ be a linearly ordered set of variables in descending order,
let $\gM_{\!\cZ}$ be an involuted monoid with $*:\ux_\tki\longleftrightarrow\uy_\tki$, $\tki\in\Jsq$;

\mat
let $\bS=\bS_{1}\times\bS_{2}\times\bS_{3}\subset\Aut(\La)$, where
$\bS_1=E$, $\bS_2= \sym\{\la_{21},\la_{22}\}$, $\bS_3= \sym\{\la_{31},\la_{32},\la_{33}\}$,
and $\bG=\bS\times\langle*\rangle \subset \Aut(\La)$;

\mat
an arbitrary transposition from $\bS$ corresponds to one of the following pairs of indices:
$\small\text{\{21,22\}}$,\, $\small\text{\{31,32\}}$,\, $\small\text{\{31,33\}}$,\, $\small\text{\{32,33\} }$,
it acts on the sets of variables $\glambda$, $\cX$, $\cY$, $\ssi$ and so on simultaneously by permuting these indices.

For the convenience of writing formulas, we introduce the notation $\bbla_\tki:=\la^*_\tki$, $\tki\in\Jsp$.

\vspace{0.1cm}

\subsubsection{The localization of $\La$ and the invariant subring}

\begin{definition}\label{def_rg_uni_formulae}
For different $\tki,\tmj\in\Jsp$ we take
$$\rg_{\tki,\tmj} = \la_\tki{+}\bbla_\tmj,\quad \text{if}\,\, |k-m|\leqslant 1,\qquad \text{for}\,\, \{i,j\}=\{1,2\};$$
and $\rg_{\tki,\tmj}=1$ for all other pairs.
\end{definition}

We denote $\brg_{\tki,\tmj}= \rg_{\tki,\tmj}^*$.
Then
\begin{align}\label{equ_brg_from_rg}
\begin{array}{lll}
\brg_{\tki,\tmj} &=& 1 - \rg_{\tki,\tmj} =-\ssi_\tki(\rg_{\tki,\tmj})=-\ssi_\tmj\mo(\rg_{\tki,\tmj}),\quad \text{if}\;\; |k-m|=1;
\\
\brg_{\tki,\tkj} &=& \phantom{1} -\ssi_\tki^2(\rg_{\tki,\tkj}),\quad \text{if}\;\; i\ne j.
\end{array}
\end{align}

Indeed, we have $\rg_{\tki,\tmj} +\brg_{\tki,\tmj} =\zeta_k+\zeta_m=1$, $|k-m|=1$, and $\rg_{\tti,\ttj}+\brg_{\tti,\ttj}=0$,
therefore $\brg_{\tki,\tmj}= 1 - (\la_\tki{+}\bbla_\tmj) =-\ssi_\tki(\la_\tki{+}\bbla_\tmj)=-\ssi_\tmj\mo(\la_\tki{+}\bbla_\tmj)$.

\vspace{0.1cm}

We consider the set
\begin{align}\label{equ_Om_prime_diskr_three}
\Om:=\big\{ \rg_{11,21},\rg_{11,22},\rg_{21,22}\big\} \bigcup {\bigcup}_{\substack{\,i=1,2,\, j=1,2,3 }}\big\{\rg_{2i,3j} \big\}
\bigcup\{\pm 1\} \subset\La.
\end{align}

\vspace{0.1cm}

Let $\gS=\gS(\Om^{\Si})$ be a $\Si$-stable multiplicative closure of the generating set $\Om$.
It follows from \eqref{equ_brg_from_rg}, $\brg_{\tki,\tmj}\in\gS$ for any $\tki, \tmj\in\Jsp$.
Then $\gS$ is the saturated, involuted, double multiplicative set,
the localization $\bL=\gS^{\!-\!1}\La$ is a partial localization with respect to the partition of $\Jsp$.

\begin{lemma}\label{lem_invariant_subring_su}
Invariant subring $\La^{\!\bG}\subset\La$, considered as a polynomial algebra, is generated by the elementary symmetric $*$-invariant polynomials
in $\glambda$ over $\bC$:
\begin{gather}\label{equ_gamma_sl_def_uni}
\begin{split}
\ga_1=\ga_1^{(1)}=\la_{11},\quad \ga_2=\ga_2^{(1)}=
\la_{21}{+}\la_{22}{-}1,\quad \ga_2^{(2)}=(\la_{21}-\ot)(\la_{22}-\ot),\\
\ga_3=\ga_3^{(1)}=\la_{31}+\la_{32}+\la_{33},\quad \ga_3^{(2)}=\la_{31}\la_{32}+\la_{32}\la_{33}+\la_{31}\la_{33},\quad \ga_3^{(3)}=\la_{31}\la_{32}\la_{33}.
\end{split}
\end{gather}
There is $\La^{\!\!\bG}=[\imath\ga_1^{(1)}, \imath\ga_2^{(1)}, \ga_2^{(2)}, \imath\ga_3^{(1)}, \ga_3^{(2)}, \imath\ga_3^{(3)}]\subset\La$.
\end{lemma}

\begin{proof}
Since $\bS$ is the direct sum of the symmetric groups according the storey partition of $\Jsp$, the invariant subring $\La^{\!\bS}\subset\La$ is generated by elementary symmetric polynomials in $\glambda$ respectively with a partition.
Any polynomial in $\glambda$ is $*$-invariant if it belongs to the subring $\bC[\imath(\la_\tki{-}\zeta_\tki{\frc{}{\!2}})]_{\,\tki\in\Jp}$ since $\big(\imath(\la_\tki{-}\zeta_\tki{\frc{}{\!2}})\big)^*=\imath(\la_\tki{-}\zeta_\tki{\frc{}{\!2}})$.
Therefore, given $\bG=\bS\times\langle*\rangle$, we can conclude that $\La^{\!\bG}$ is a ring of symmetric polynomials
separately for each of the groups of variables
$\imath\la_{11}$, $\imath (\la_{2i}{-}\ot)$, $i=1,2$, and $\imath\la_{3i}$, $i=1,2,3$.
\end{proof}


\begin{remark}
Notice, that $\ttd$ is multi symmetric with respect to the given partition of $\Jp$ and $\ttd^*=\ttd$, hence $\ttd\in\La^{\!\!\bG}$,
moreover $\ttd^{\ssi}\cdot\ttd^{\ssi\mo} \in\La^{\!\!\bG}$ for any $\ssi\in\Si$.
Denote $\LL^{\!\!\bG}\subset\LL$ the invariant subring.
Any element from $\LL^{\!\!\bG}$ is a finite product of the elements a type:\,\,\,
$F=\oprod_{\ssi\in\Si}\dfrac{1}{\ttd^{\ssi}\,\ttd^{\ssi\mo}}\cdot\ff,\,\, \text{where}\,\, \ff\in\La^{\!\!\bG}.$
\end{remark}

\vspace{0.1cm}

\subsubsection{The structure parameters $\ft$, $\vmu$}

\smallskip~
\smallskip~

\begin{definition}\label{den_tt_bt}
We define the elements $\tt_\tki,\bt_\tki\in\LLt$ as follows:
\begin{align}\label{def_tt_bt_uni}
\begin{array}{ll}
\tt_{11}=\rg_{11,21}\rg_{11,22}, & \bt_{11}=\brg_{11,21}\brg_{11,22};
\\
\tt_{2i}=\dfrac{\rg_{2i,11}\, {\prod}_{k=1}^{3}\rg_{2i,3k}}{\rg_{2i,2j}\,\ssi_\tti(\rg_{\tti,\ttj})}, &
\bt_{2i}=\dfrac{\brg_{2i,11}\, {\prod}_{k=1}^{3}\brg_{2i,3k}}{\brg_{2i,2j}\,\ssi\mo_\tti(\brg_{\tti,\ttj})},\quad
\{i,j\}=\{1,2\}.
\end{array}
\end{align}

By \eqref{equ_brg_from_rg}, $\bt_\tki = \ssi_\tki(\tt_\tki)$,
and $\ssi(\tt_\tki),\ssi(\bt_\tki)\in\LLt$ for any $\tki\in\Jsq$, $\ssi\in\Si$.
\end{definition}

To determine the matrix $\vmu$, we put: $\muxy_{\tki,\tmj}=\muyx_{\tki,\tmj}=1$\,
for all\, $\tki,\tmj\in\Jsq$,\, $\tki\ne\tmj$.
Thereafter, we use equalities \eqref{lem_muxy_condition} to determine the remaining parameters as follows:
for any $\tki,\tmj\in\Jq$ with $|k-m|\leqslant1$, $\tki\ne\tmj$, we take
\begin{gather}\label{equ_mu_tt_dependence_uni}
\begin{split}
\muxx_{\tki,\tmj} = \dfrac{\ssi\mo_{\tmj}(\tt_{\!\tki})}{\tt_{\!\tki}} =\dfrac{\bt_{\!\tki}}{\ssi_{\tmj}(\bt_{\!\tki})} =\dfrac{\ssi_{\!\tki}(\bt_{\tmj})}{\bt_{\tmj}} =\dfrac{\tt_{\tmj}}{\ssi\mo_{\!\tki}(\tt_{\tmj})},
\qquad
\dfrac{\tt_\tki}{\ssi_\tmj(\tt_\tki)} = \dfrac{\bt_\tmj}{\ssi\mo_\tki(\bt_\tmj)}.
\end{split}
\end{gather}
\noindent
Besides, $\muyy_{\tki,\tmj}=\big(\muxx_{\tki,\tmj}\big)\st$,
$\muxx_{\tki,\tmj}\muxx_{\tmj,\tki}=\muxx_{\tki,\tmj}\muyy_{\tki,\tmj}=1$ for any indexes $\tki,\tmj$.

\begin{corollary}
The obtained parameter values
$\muxx_{\tki,\tmj}{=}-\dfrac{\brg_{\tki,\tmj}}{\rg_{\tki,\tmj}}{=}-\dfrac{\bbla_\tki+\la_\tmj}{\la_\tki+\bbla_\tmj}$,\,
$\tki,\tmj\in\Jp$,\, $\tki{\ne}\tmj$,\, $|k{-}m|=1$, and
$\muxx_{21,22}{=} -\dfrac{\rg_{21,22}}{\brg_{21,22}}{=}-\dfrac{\la_{21}+\bbla_{22}}{\bbla_{21}+\la_{22}}$
satisfy \eqref{equ_mu_conditions_one}, \eqref{equ_mu_conditions_two}.
\end{corollary}

\smallskip~

\subsubsection{The subalgebra $\fA_\sl\subset\fA$}

\smallskip~
\smallskip~

In this subsection, given the RTGW algebra $\fA=\Alg_{3,6}\big(\gS,\bG,\vsigma,\ft,\vmu\big)$,
we establish an isomorphism between the universal enveloping algebra $\op{U}(\sl(3,\bC))$ of the matrix special linear Lie algebra $\sl(3;\bC)$
and a subalgebra $\fA_\sl$ of the Lie algebra $\fA$, defined below.

\smallskip~

\begin{lemma}\label{lem_Eij_Aij_correspond}
We denote by $\fA'\subset\fA$ a $\bC$-span of the following elements:
\begin{gather}\label{equ_def_sl_XY}
\begin{split}
\begin{array}{ll}
\gX_1=\ux_{11},\,\,\,\gY_1=\uy_{11}, &\gX_2=\ux_{21}+\ux_{22},\,\,\, \gY_2=\uy_{21}+\uy_{22},
\\
\gX_{3}=[\gX_{2},\gX_{1}] =\vrho_1\,\ux_{21}\ux_{11}+\vrho_2\,\ux_{22}\ux_{11}, &
\gY_{3}=[\gY_{1},\gY_{2}] = \vrho_1\,\uy_{11}\uy_{21}+\vrho_2\,\uy_{11}\uy_{22},
\\
\gH_1=2\ga_1-\ga_2,\,\, \gH_2=2\ga_2-\ga_1-\ga_3, & \gH_3=\ga_3-\ga_1-\ga_2=-\gH_1-\gH_2,
\\
\end{array}
\end{split}
\end{gather}
where
$\vrho_i=(1-\muxx_{11,2i})=(1-\muyy_{2i,11})=(\la_{11}+\bbla_{2i})^{-1}$.
There hold
\begin{gather}\label{equ_mu_from_eta}
\muxx_{11,2i}=\muyy_{2i,11}=-\dfrac{\vrho_i}{\vrho_i^*},\quad\quad
\muyy_{11,2i}=\muxx_{2i,11}=-\dfrac{\vrho_i^*}{\vrho_i}.
\end{gather}

\smallskip~

In this case, the multiplication table relative to the Lie brackets is as follows:

$$\begin{tabular}{r|rrrrrrrrrrr} 
\hline
$[\,,\,]$ & $\gX_1$ & $\gY_1$ & $\gX_2$ & $\gY_2$ & $\gX_3$ & $\gY_3$ & $\gH_1$ & $\gH_2$  & $\phantom{I^{I^I}}$ \\
\hline
$\gY_1$ & $\gH_{1}$ & $0\;$ & $0\;$ & $\gY_3$ & $-\gX_2$ & $0\;$ & $ 2\gY_1$ & $-\gY_1$ & $\phantom{I^{I^I}}$
\\
$\gX_1$ & $~$ & $-\gH_{1}$ & $-\gX_3$ & $0\;$ & $0\;$ & $\gY_2$ & $ -2\gX_1$ & $\gX_1$ & $\phantom{I^{I^I}}$
\\
$\gY_2$ & $~$ & $~$ & $\gH_{2}$ & $0\;$ & $\gX_1$ & $0\;$ & $-\gY_2$ & $2\gY_2$ & $\phantom{I^{I^I}}$
\\
$\gX_2$ & $~$ & $~$ & $~$ & $-\gH_{2}$ & $0\;$ & $-\gY_1$ & $\gX_2$ & $-2\gX_2$ & $\phantom{I^{I^I}}$
\\
$\gY_3$ & $~$ & $~$ & $~$ & $~$ & $\gH_3$ & $0\;$ & $ \gY_3$ & $\gY_3$ & $\phantom{I^{I^I}}$
\\
$\gX_3$ & $~$ & $~$ & $~$ & $~$ & $~$ & $-\gH_3$ & $-\gX_3$ & $-\gX_3$ & $\phantom{I^{I^I}}$
\\
\end{tabular}$$

\smallskip~

The elements $\gH_i,\gX_i,\gY_i$ are invariants of the group $\bS$ for $i=1,2$, while $\gX_3,\gY_3$ are anti-invariants:
$(\gX_3)^{\pi_{21,22}}=-\gX_3$.

The involution $*$ acts on the generator by the formulae:
\begin{align}\label{equ_action_involution_su}
(\gH_i)^*=\gH_i,\quad (\gX_1)^*=\gY_1,\quad (\gX)_2^*=\gY_2,\quad (\gX_3)^*=-\gY_3.
\end{align}

Moreover, we take $(\gX_1)^*=\gY_1$, $(\gX_2)^*=\gY_2$, and obtain $(\gX_3)^*=-\gY_3$.

Then $\fA'\subset\fA$ is an involuted subalgebra generated over $\bC$ by the elements $\gX_{1},\, \gY_{1},\, \gX_{2},\, \gY_{2}$.
\end{lemma}

\smallskip~

\begin{proof}
The proof of the lemma consists in a direct verification of the multiplication equalities.
We will only mention a few more complicated ones.
We take into account the fact that the elements $\uy_\tki$ and $\ux_\tmj$ with different indices commute to each other.
First, just check that

$${[}\gY_{\!1},\gX_{1}{]}{=}\tt_{11}{-}\bt_{11}{=}\gH_1,\quad
{[}\gY_{2},\gX_{2}{]} = \tt_{21}{+}\tt_{22}{-}\bt_{21}{-}\bt_{22} = \gH_2,\quad
{[}\gY_{\!1},\gY_{2}{]} = \gY_3,\quad {[}\gX_{1},\gX_{2}{]} = -\gX_3.$$

\clearpage

Next, we obtain some auxiliary statements, proceeding in steps.

\begin{step}\label{step_02}
For $i=1,2$ the following consistency conditions hold:
\begin{gather}\label{equ_basic_products_two}
\begin{split}
\begin{array}{rcrccrcr}
\vrho_i\,\uy_{11}\uy_{2i}\cdot\vrho_i\,\ux_{2i}\ux_{11} &{=}&  -\vrho_i\vrho_i^*\cdot\tt_{11}\tt_{2i};
&&
\vrho_i\,\ux_{2i}\ux_{11}\cdot\vrho_i\,\uy_{11}\uy_{2i} &{=}&  -\vrho_i\vrho_i^*\cdot\bt_{11}\bt_{2i};
\\
\vrho_1\,\uy_{11}\uy_{21}\!\cdot\!\vrho_2\,\ux_{22}\ux_{11} &{=}& \vrho_1\vrho_2\tt_{11}\!\cdot\!\uy_{21}\ux_{22};
&&
\vrho_2\,\ux_{22}\ux_{11}\!\cdot\!\vrho_1\,\uy_{11}\uy_{21} &{=}& \vrho^*_1\vrho^*_2\bt_{11}\!\cdot\!\uy_{21}\ux_{22}.
\end{array}
\end{split}
\end{gather}
\end{step}

\begin{proof}
We calculate $\vrho_i^*=1-\muxx_{2i,11}=(\bbla_{11}{+}\la_{2i})^{-1}$. Then

$\begin{array}{llll}
\vrho_i\,\uy_{11}\uy_{2i}\,\vrho_i\,\ux_{2i}\ux_{11} &= \vrho_i\,\ssi\mo_{11}\ssi\mo_{2i}(\vrho_i)\muyy_{11,2i}\tt_{11}\tt_{2i} =-\vrho_i\vrho_i^*\tt_{11}\tt_{2i};
\\
\vrho_1\,\uy_{11}\uy_{21}\cdot\vrho_2\,\ux_{22}\ux_{11}
&= \vrho_1\,\ssi\mo_{11}\ssi\mo_{21}(\vrho_2)\ssi\mo_{11}\ssi\mo_{21}(\muxx_{22,11})\tt_{11}
\uy_{21}\ux_{22} = \vrho_1\vrho_2\tt_{11}\cdot\uy_{21}\ux_{22};
\\
\vrho_2\,\ux_{22}\ux_{11}\cdot\vrho_1\,\uy_{11}\uy_{21}
&= \vrho_2\,\ssi_{11}\ssi_{22}(\vrho_1)\muxx_{22,11}\bt_{11}\uy_{21}\ux_{22}
= \vrho^*_1\vrho^*_2\bt_{11}\cdot\uy_{21}\ux_{22}.\\
\end{array}$
\end{proof}

\begin{step}\label{step_03}
The equality $[\gY_3,\gX_3] = \gH_3$ is fulfilled, which follows from the relations:

\noindent
1) ${\sum}_{i=1,2}[\vrho_i\,\uy_{11}\uy_{2i},\,\vrho_i\ux_{2i}\ux_{11}]{=}-\gH_3,$\quad
2)\, ${[}\vrho_1\,\uy_{11}\uy_{21},\,\vrho_2\,\ux_{22}\ux_{11}{]} = {[}\vrho_2\,\uy_{11}\uy_{22},\,\vrho_1\,\ux_{21}\ux_{11}{]}{=}0.$
\end{step}

\begin{proof}
1)
$\sum_{i=1,2}{[}\vrho_i\,\uy_{11}\uy_{2i},\,\vrho_i\,\ux_{2i}\ux_{11}{]}
\overset{\eqref{equ_basic_products_two}}{=}-\sum_{i=1,2}\vrho_i\vrho_i^*\big(\tt_{11}\tt_{2i}-\bt_{11}\bt_{2i}\big)= \ga_3-\ga_2-\ga_1 =\gH_3$.
\begin{gather*}
\begin{array}{llll}
2)\quad {[}\vrho_1\,\uy_{11}\uy_{21},\,\vrho_2\,\ux_{22}\ux_{11}{]}
&= \big(\vrho_1\,\ssi\mo_{11}\ssi\mo_{21}(\vrho_2)\ssi\mo_{11}\ssi\mo_{21}(\muxx_{22,11})\tt_{11}
-\vrho_2\,\ssi_{11}\ssi_{22}(\vrho_1)\muxx_{22,11}\bt_{11}\big)\uy_{21}\ux_{22}
\\
&= \big(\vrho_1\vrho_2\tt_{11}-\vrho^*_1\vrho^*_2\bt_{11}\big)\cdot\uy_{21}\ux_{22} =0.
\end{array}
\end{gather*}

Then we have
$[\gY_3,\gX_3] =\osum_{i=1,2}[\vrho_i\,\uy_{11}\uy_{2i},\,\vrho_i\ux_{2i}\ux_{11}]\,
= -\osum_{i=1,2}\vrho_i\vrho_i^*\big(\tt_{11}\tt_{2i}-\bt_{11}\bt_{2i}\big) = \gH_3$.
\end{proof}

\begin{step}\label{step_04}
There holds $(\gX_3)^*=-\gY_3$.
\end{step}

\begin{proof}
We get
$(\vrho\is\,\uy_{11}\uy_{2i})^*{=}\vrho\is^*\,\ux_{11}\ux_{2i}{=}(1{-}\muxx_{11,2i})^*\muxx_{11,2i}\,\ux_{2i}\ux_{11}
{=}(\muxx_{11,2i}{-}1)\,\ux_{2i}\ux_{11}{=}-\vrho\is\ux_{2i}\ux_{11}$,
and $(\vrho_1\,\ux_{21}\ux_{11})^* = -\vrho_1\uy_{11}\uy_{21}$, similarly.
\end{proof}

\noindent
The remaining equalities can be verified similarly.
This completes the proof of Lemma \ref{lem_Eij_Aij_correspond}.
\end{proof}

\vspace{0.1cm}

\subsection{The correspondence with matrix algebras}

\smallskip~
\smallskip~

Remember that $E_{kl}\in M_{3\times 3}(\bC)$,$1\leqslant k,l\leqslant 3$ generates the matrix special linear Lie algebra $\sl(3,\bC)$.

\begin{lemma}\label{lem_Eij_Aij_correspond}
The linear map $\op{U}(\sl(3,\bC))\longrightarrow\fA'$, defined on basis elements by correspondences
\begin{gather}\label{equ_Eij_Aij_correspond}
\begin{array}{lccclccclccc}
E_{21} &\rightarrow& \gX_{1},
&\quad&
E_{12} &\rightarrow& \gY_{\!1},
&\quad&
E_{11}-E_{22} &\rightarrow& \gH_1,
\\
E_{32} &\rightarrow& \gX_{2},
&&
E_{23} &\rightarrow& \gY_{2},
&&
E_{22}-E_{33} &\rightarrow& \gH_2,
\\
E_{31} &\rightarrow& \gX_{3},
&&
E_{13} &\rightarrow& \gY_{3},
&&
E_{11}-E_{33} &\rightarrow& \gH_3. \\
\end{array}
\end{gather}
preserves the Lie bracket operations and gives an embedding of $\op{U}(\sl(3,\bC))$ to $\fA$.
In such a way $\op{U}(\sl(3,\bC))$ is isomorphic to the subalgebra $\fA'\subset\fA$ generated by the elements
$\gX_{1},\, \gY_{1},\, \gX_{2},\, \gY_{2}$, \eqref{equ_def_sl_XY}.
\end{lemma}

Here $E_{12}^*=E_{21}$, $E_{23}^*=E_{32}$, and $E_{13}^*=[E_{12},E_{23}]^*=[E_{21},E_{32}]=-E_{31}$.

The proof is obtained by direct computation using the multiplication table.
The statement about the isomorphism is due to the simplicity of the algebra $\op{U}(\sl(3,\bC))$.
From now on, we will use the notation $\fA_{\sl}=\fA'$.

\smallskip~

\subsubsection{The invariant subalgebra $\fA_\su\subset\fA_\sl$}

\smallskip~
\smallskip~

\begin{denotations}\label{def_su_generators}
Within the RTGW algebra $\fA_\sl$ we consider the elements
\begin{gather*}
\begin{array}{llllllll}
\cU_1 = & \gX_{1}+\gY_{1}, &&
\cV_1 = & {\imath}(\gY_{1}-\gX_{1}), &&
\cH_1 = & \imath\gH_1,
\\
\cU_2=& \gX_{2}+\gY_{2}, &&
\cV_2=& {\imath}(\gY_{2} -\gX_{2}), &&
\cH_2 = & \imath\gH_2,
\\
\cU_3 =& \gY_{3} -\gX_{3}, &&
\cV_3 =& {\imath}(\gY_{3} +\gX_{3}) &&
\cH_3 = & \imath\gH_3.
\end{array}
\qquad\text{(see \eqref{equ_def_sl_XY}).}
\end{gather*}
Denote by $\fA''$ the RTGW subalgebra of $\fA_{\sl}$ generated by the elements $\{\cU_k,\cV_k,\cH_k,\,\, k=1,2,3\}$.
\end{denotations}


\begin{lemma}\label{lem_su_generators}
The algebra $\fA''$ is an invariant subalgebra of $\fA_\sl$ with respect to the group action, $\fA''\subset\fA_\sl^{\!\bG}\subset\fA^{\!\bG}$.
\end{lemma}

\begin{proof}
Since the group $\bG$ acts on $\fA$ multiplicatively, it suffices to prove the invariance condition only for generating elements,
and it is a consequence of Lemma \ref{lem_Eij_Aij_correspond}, \eqref{equ_action_involution_su}
using $\imath^*=-\imath$.
\end{proof}

\vspace{0.1cm}

\begin{lemma}\label{lem_su_RTGWA}
Element-wise correspondence of the generators
\begin{gather}\label{equ_GellMann_correspondence}
\begin{array}{ccccccccccccccccc}
F_1   &&    F_2 &&   F_6 &&   F_7 &&  F_5 &&   F_4 &&    F_3 &&   F_8
&&
\\
\updownarrow &&  \updownarrow && \updownarrow && \updownarrow && \updownarrow && \updownarrow && \updownarrow && \updownarrow
\\
\cU_1 && \cV_1 &&  \cU_2 && \cV_2 &&    -\imath\cU_3 &&        \imath\cV_3 && -\imath\cH_1 && \imath/_{\!\!\sqrt{3}}\,(\cH_1+2\cH_2) &&
\\
\end{array}
\end{gather}
establishes an isomorphism between the universal enveloping $\op{U}(\su(3;\bC))$ of the complex special unitary Lie algebra $\su(3;\bC)$ presented by the Gell-Mann generators \eqref{equ_Gell_Mann_matrix}, and the RTGW algebra $\fA''$.
Besides, there are $\cH_2\leftrightarrow {-}\imath{\frc{}{\!\!2}} (F_3{-}\sqrt{3}F_8)$,
and $\cH_3\leftrightarrow {-}\imath{\frc{}{\!\!2}} (F_3{+}\sqrt{3}F_8)$.
\end{lemma}

\begin{proof}
To verify the existence of a required isomorphism, it suffices to verify that the Lie product of any two generating elements of $\fA''$
corresponds to the Lie product of the corresponding matrices.
It is easy to verify directly that
\begin{gather*}
\begin{array}{ccccccc}
{[}\cU_1,\cU_2{]} = -{[}\cV_1,\cV_2{]} =\cU_3, && {[}\cU_1,\cV_2{]} ={[}\cV_1,\cU_2{]} =-\cV_3,
&& {[}\cU_1,\cV_1{]} =2\cH_1, && {[}\cU_2,\cV_2{]} =2\cH_2,
\\
{[}F_1,F_6{]} = -{[}F_2,F_7{]} =\imath F_5, && {[}F_1,F_7{]} =-{[}F_2,F_6{]} =\imath F_4,
&& {[}F_1,F_2{]} = 2\imath F_3, && {[}F_6,F_7{]},
\end{array}
\end{gather*}
where ${[}F_6,F_7{]} = -2\imath (F_3-\sqrt{3}F_8)$.
\end{proof}

Therefore, we have reason to re-designate the algebra under consideration in the form $\fA_\su=\fA''$.

\vspace{0.1cm}

Theorem \ref{the_su_RTGWA} follows obviously from Lemma \ref{lem_su_RTGWA}.

\begin{corollary}
The above relations define an isomorphism between algebra $\su(3;\bC)$ and Lie algebra $\fg_\su(3;\bC)$ generated by elements on the bottom row of the correspondence \eqref{equ_GellMann_correspondence}, and the universal embedding of the algebra $\fg_\su(3;\bC)$ belongs to RTGWA $\fA_{\su}$.
\end{corollary}

\vspace{0.1cm}

\subsubsection{Casimir operators of $\fg_\su(3;\bC)$}

\smallskip~
\smallskip~

Correspondence \eqref{equ_GellMann_correspondence} allows one to describe the Cartan subalgebra
and calculate the Casimir operator for the algebra $\fg_\su(3;\bC)$ without additional proofs.

The Cartan subalgebra of Lie algebra $\fg$ is considered as nilpotent subalgebra of $\fg$, which is equal to its normalizer in $\fg$.


\begin{remark}
The Cartan subalgebra of $\fg_\su(3;\bC)$
considered as nilpotent subalgebra of $\fg_\su(3;\bC)$
is a two-dimensional subalgebra $\mathfrak{h}$ generated by the elements
$\cH_1 = 2\ga_1-\ga_2$, and $\cH_2 = 2\ga_2-\ga_1-\ga_3$.
Then $\mathfrak{h}$  is maximal abelian self-normalising subalgebra of the Lie algebra $\fg_\su(3;\bC)$.
\end{remark}

A Casimir element (also known as a Casimir invariant or Casimir operator) is a distinguished element of the center of the enveloping algebra of a Lie algebra.
In the case of $\SU(3)$ group and of $\su(3;\bC)$ Lie algebra, two independent Casimir operators can be constructed, a quadratic and a cubic:
they have the following form:
$C_2=\dfrac{1}{4}{\sum} _{k}{F_{k}}{F_{k}}$,\, ${{C_3}}=\dfrac{1}{8}{\sum} _{{jkl}}d_{{jkl}}{{F_{j}}}{{F_{k}}}{{F_{l}}}~,$\,
where $d_{jkl}$ are symmetric under the permutation of any pair of indices
non-zero independent elements of the tensor $d_{jkl}$ of $\su(3;\bC)$,
the non trivial elements take the values
\begin{gather*}
d_{146} =d_{157} =-d_{247} =d_{256} =1/{2};
\quad
d_{344} =d_{355} =-d_{366} =-d_{377} =1/{2};
\\
d_{118} =d_{228} =d_{338} =-d_{888} =1/{\sqrt{3}};  
\quad
d_{448} =d_{558} =d_{668} =d_{778} = -1/{2\sqrt{3}}. 
\end{gather*}

\vspace{0.2cm}

\begin{lemma}\label{lemmma_Casimir_formulae}
The quadratic and cubic Casimir elements of $\fg_\su(3;\bC)$ belong to the center of the embedding algebra $U\!(\fg_\su(3;\bC))$,
they have the following form in parameters of \eqref{equ_gamma_sl_def_uni}:
\begin{align}\label{equ_Casimir_formulae}
C_2 = {\frc{1\!}{\!3}}\;\big((\ga_{3}^{(1)})^2-3\ga_{3}^{(2)}-3\big);
\qquad
C_3 ={\imath\,\frc{\!}{\!18}}\;
\big( 9\ga_{3}^{(2)}\,\ga_{3}^{(1)} -2\big(\ga_{3}^{(1)}\big)^3-{27}\ga_{3}^{(3)})\big).
\end{align}
\end{lemma}

\vspace{0.1cm}

\begin{proof}
First, we calculate the anti-commutator
${\{}\vrho_1\,\uy_{11}\uy_{21},\,\vrho_2\,\ux_{22}\ux_{11}{\}}
=\big(\vrho_1\vrho_2\tt_{11}{+}\vrho^*_1\vrho^*_2\bt_{11}\big)\cdot\uy_{21}\ux_{22} ={-}2\cdot\uy_{21}\ux_{22}$,
and obtain
\begin{gather}\label{equ_Psi_calculation}
{\{}\vrho_1\,\uy_{11}\uy_{21},\,\vrho_2\,\ux_{22}\ux_{11}{\}}+{\{}\vrho_2\,\uy_{11}\uy_{22},\,\vrho_1\,\ux_{21}\ux_{11}{\}}
=-2(\uy_{21}\ux_{22}+\uy_{22}\ux_{21}).
\end{gather}

\smallskip~

The proof of the formulae \eqref{equ_Casimir_formulae} consists in directly calculating of this sum, divided into several terms.
By the formulae, $\sum _{k}{F_{k}}{F_{k}} \longleftrightarrow C'_2=4 C_2$.
Let $\Psi{=}\uy_{21}\ux_{22}{+}\uy_{22}\ux_{21}$.
Using \eqref{equ_Psi_calculation}, we get
\begin{gather*}
\begin{array}{cclcl}
C'_2 &=& \big(\cU_1^2+\cV_1^2\big) +\big(\cU_2^2+\cV_2^2\big) -\big(\cU_3^2+\cV_3^2\big)-\big(\cH_1^2+(\cH_1+2\cH_2)^2/3\big)=
\\
&=& (\bt_{11}-\tt_{11})^2 - 2(\tt_{11}+\bt_{11})
\,\,\, + (\tt_{11}-\bt_{11}+2(\tt_{21}+\tt_{22}-\bt_{21}-\bt_{22}))^2/3
\\
&-& 2{\sum}_{i=1,2}(\tt_{2i}+\bt_{2i}) -4\Psi
\,\,\, - 2{\sum}_{i=1,2}\vrho_i\vrho^*_i(\tt_{11}\tt_{2i}+\bt_{11}\bt_{2i}) +4\Psi
\\
&=& {\frc{4\!}{\!3}}\,\big((\la_{31}^2+\la_{32}^2+\la_{33}^2)-(\la_{31}\la_{32}+\la_{31}\la_{33}+\la_{32}\la_{33})-3\big)
={\frc{4\!}{\!3}}\,\big((\ga_3^{(1)})^2-3\ga_3^{(2)}-3\big).
\end{array}
\end{gather*}

The calculation for $C_3$ is technically more difficult and is obtained by computer calculation.
\end{proof}

\vspace{0.2cm}

\subsection{The enveloping algebras of orthogonal Lie algebra as rational TGWA}

\subsubsection{Structure datum}

\smallskip~

Let $\gq=2$, $\gp=4$. We consider the following datum $\big(\gS,\bG,\vsigma,\ft,\vmu\big)$:

\mat
we denote by $\Jsq=\{\small\text{{11,\,21}}\}$ and $\Jsp=\{\small\text{{11,\,21,\,31,\,32}}\}$ the sets of indices
given together with the storey partition
$\Jsp=J_1\cup J_2\cup J_3=\{\small\text{{11}}\} \cup \{\small\text{{21}}\}\cup\{\small\text{{31,\,32}}\}$;

\mat
let $\glambda=\{\la_\tki\}_{\tki\in\Jp}$, $\gzeta=\{\gzeta_\tki\}_{\tki\in\Jp}$
where $\gzeta_{11}=0$, $\gzeta_{21}=1$, $\gzeta_{31}{=}\gzeta_{32}{=}0$, and let
$\La=\bC[\,\glambda\,]^{\gzeta}$ be an involuted polynomial ring;

\mat
let $*\in\Aut(\La)$ be an involution on $\La$ such that: $\la_\tki^*=\zeta_\tki-\la_\tki$, $\tki\in\Jsq$, and $\imath^*=-\imath$;

\mat
let $\bS=\bS_{1}\times\bS_{2}\times\bS_{3}\subset\Aut(\La)$ acts on $\glambda$ by the mutually independent permutations
of the groups of variables $\{\la_{11}\}$, $\{\la_{21}\}$ and $\{\la_{31},\la_{32}\}$,
hence $\bS\simeq \langle \pi_{31,32} \rangle \simeq\sym_2$;

\mat
let $\bG=\bS\times\langle*\rangle \subset \Aut(\La)$ denotes the automorphism group;

\mat
let $\Si\in\Aut(\La)$ be an abelian group of shifts, $\Si=\langle\ssi_{11},\ssi_{21}\rangle \simeq \bZ^{\otimes{2}}$,
under the actions on $\La$ as:\, $\ssi_\tki(\la_\tki)=\la_\tki-1$,\, $\ssi_\tki(\la_\tmj)=\la_\tmj$, $\tmj\ne\tki$;

\mat
let $\cZ=\{\ux_{11},\uy_{11},\ux_{21},\uy_{21}\}$ be the variable set of monoids $\gM_{\!\cZ}$ and $\bgM_{\!\cZ}$,
an involution $*$ acts according to the rule $*:\ux_{11}\longleftrightarrow\uy_{11}$ and $*:\ux_{21}\longleftrightarrow\uy_{21}$.

\begin{lemma}\label{lem_invariant_subring_so}
The invariant subring $\Gamma=\La^{\!\!\bG}\subset\La$ is spanned by the elementary symmetric $*$-invariant polynomials in $\glambda$ over $\bC$:
$\imath\ga_i^{(1)}$, $i=1,2,3$ and $\ga_3^{(2)}$ where
\begin{gather}\label{equ_gamma_sl_def_ort}
\begin{split}
\ga_1^{(1)}\!=\!\la_{11},\quad \ga_2^{(1)}\!=\!\la_{21}{-}\ot\,,\quad
\ga_3^{(1)}=\la_{31}+\la_{32},\quad \ga_3^{(2)}=\la_{31}\la_{32}.
\end{split}
\end{gather}
So, $\Gamma=\bR\,[\imath\ga_1^{(1)}, \imath\ga_2^{(1)}, \imath\ga_3^{(1)}, \ga_3^{(2)}]\subset\La^{\bG}$.
\end{lemma}

The proof is similar to that of the Lemma \ref{lem_invariant_subring_su}.
The defining polynomial $\ttd$ corresponding the given partition belongs to $\Gamma$,
besides, $\ssi(\ttd)\ssi\mo(\ttd)\in\Gamma$ for any $\ssi\in\Si$.


\subsubsection{The structural constants and the localization of $\La$}

\smallskip~

Below we define the structural parameters in the considered case of RTGW algebra.

\begin{definition}\label{def_rg_ort_formulae}
For $\tki,\tmj\in\Jsp$ we put
\begin{gather*}
\begin{array}{llll}
\rg_{\tki,\tmj}{=}(\la_\tki+\la_\tmj)(\la_\tki+\bbla_\tmj), &\quad&
\brg_{\tki,\tmj}{=}(\bbla_\tki+\la_\tmj)(\bbla_\tki+\bbla_\tmj), & |k{-}m|=1;
\\
h_{21}=\la_{21}^2(\la_{21}{-}\bbla_{21})\ssi_{21}\mo(\la_{21}{-}\bbla_{21}),
&\quad& \bar{h}_{21}=\bbla_{21}^2(\bbla_{21}{-}\la_{21})\ssi_{21}(\bbla_{21}{-}\la_{21}).
\end{array}
\end{gather*}

We put $\rg_{\tki;\tmj}=1$ for all other pairs.
For $\{\tki,\tmj\}=\{\small\text{{11,\,21}}\}$, we take
\begin{align}\label{equ_def_mu_ort}
\begin{split}
\begin{array}{llllll}
\muxx_{\tki,\tmj}&=&-\dfrac{\rg_{\tmj,\tki}}{\rg_{\tki,\tmj}} =-\dfrac{(\la_\tmj+\bbla_\tki)}{(\la_\tki+\bbla_\tmj)},
&\quad
\muxy_{\tki,\tmj}&=&-\dfrac{\brg_{\tmj,\tki}}{\rg_{\tki,\tmj}} =-\dfrac{(\bbla_\tmj+\bbla_\tki)}{(\la_\tki+\la_\tmj)},
\\
\muyx_{\tki,\tmj}&=&-\dfrac{\rg_{\tmj,\tki}}{\brg_{\tki,\tmj}}  =-\dfrac{(\la_\tmj+\la_\tki)}{(\bbla_\tki+\bbla_\tmj)},
&\quad
\muyy_{\tki,\tmj}&=&-\dfrac{\brg_{\tmj,\tki}}{\brg_{\tki,\tmj}}  =-\dfrac{(\bbla_\tmj+\la_\tki)}{(\bbla_\tki+\la_\tmj)},
\end{array}
\end{split}
\end{align}
and suppose these $\vmu$-parameters obey the skew commutation laws \eqref{equ_mu_commutant_two_new}.
\end{definition}

Clearly, $\rg_{\tki,\tmj}^*$ $=\brg_{\tki,\tmj}{=}\ssi_\tki(\rg_{\tki,\tmj})$,\,
and therefore $\muyy_{\tki,\tmj}{=}(\muxx_{\tki,\tmj})^*$, $\muyx_{\tki,\tmj}{=}(\muxy_{\tki,\tmj})^*$.

\smallskip~

Let $\gS=\gS(\Om^{\bG})$ be a $*,\Si,\times$-stable multiplicative closure of the generating set
\begin{align}\label{equ_Om_prime_diskr_three}
\Om:=\big\{\pm1,\, \la_{21},\, (\la_{21}{+}\la_\tki),\, (\la_{21}{-}\la_\tki),\, (2\la_{21}{-}1) \big\} \subset\La,\quad
\tki\in\{\small\text{{11,\,31,\,32}}\}.
\end{align}
Then $\gS$ is a $\Si$-finite involuted two-parameter multiplicative set, and $\ssi(\rg_{\tki,\tmj})\in\gS$
for any $\tki,\tmj\in\Jp$ and any $\ssi\in\Si$.
Denote by $\LL=\gS\mo\La$ localization of $\La$ with $\gS$.

\vspace{0.2cm}

\begin{definition}\label{def_tt_bt}
We define the set $\ft=\{\tt_\tki\}_{\tki\in\Jq}$ as follows:
\begin{align}\label{equ_def_tt_ort}
\begin{array}{cclcccl}
\tt_{11}&=&{\rg_{{11},21}}/{4}; && \bt_{11}&=&{\brg_{{11},21}}/{4};
\\
\tt_{21}&=&{\rg_{21,11}\rg_{21,31}\rg_{21,32}}/{h_{21}};
&&
\bt_{21}&=&{\brg_{21,11}\brg_{21,31}\brg_{21,32}}/{\bar{h}_{21}}.
\end{array}
\end{align}
\end{definition}
By definition, $\ssi(\tt_\tki),\ssi(\tt_\tki\mo){\in}\LL$ for any $\tki{\in}\Jq$ and any $\ssi{\in}\Si$,
and, besides, $\ssi_\tki(\tt_\tki){=}\bt_\tki$.

A direct verification shows that conditions \eqref{equ_mu_conditions_one}, \eqref{equ_mu_conditions_two}
and \eqref{equ_mu_conditions_three} are satisfied.

\vspace{0.1cm}

We denote by $\fA=\Alg_{2,4}\big(\gS,\bG,\vsigma,\ft,\vmu\big)$ the RTGW algebra with the datum $\big(\gS,\bG,\vsigma,\ft,\vmu\big)$
relative to the monoid $\gM_{\!\cZ}$.

\vspace{0.2cm}

\subsubsection{The subalgebra $\fA_\so\subset\fA$}

\begin{lemma}\label{lem_so_generators}
The elements $\cU_1,\cU_2,\cU_3$ defined below belong to the invariant subalgebra $\fA^{\!\bG}$:
\begin{align}\label{formula_GZ-3-delta-full}
\begin{array}{l}
\cU_1=\ux_{11}+\uy_{11},  \\
\cU_2=\ux_{21}+\uy_{21}+\cC, \\
\cU_3=\vrho^{xx}\ux_{11}\ux_{21} +\vrho^{xy}\ux_{11}\uy_{21} +\vrho^{yx}\uy_{11}\ux_{21} +\vrho^{yy}\uy_{11}\uy_{21}
-{\la_{11}\mo}{\cC}(\ux_{11}-\uy_{11}),
\end{array}
\end{align}
where $\cC\!=\!-i\dfrac{\la_{11}\la_{31}\la_{32}}{\la_{21}(1{-}\la_{21})}$
and $\vrho^{xx}\!=\!1\!-\!\mu^{xx}_{21,11}$,\,
$\vrho^{yx}\!=\!1\!-\!\mu^{xy}_{21,11}$,\,
$\vrho^{xy}\!=\!1\!-\!\mu^{yx}_{21,11}$,\,
$\vrho^{yy}\!=\!1\!-\!\mu^{yy}_{21,11}$.

The commutation relations of the elements $\cU_1,\cU_2,\cU_3$ in $\fA$ are of the form:
\begin{align}\label{equ_ort_bracket_one}
[\cU_1,\cU_2]=\cU_3,\qquad  [\cU_3,\cU_1]=\cU_2,\qquad [\cU_3,\cU_2]=-\cU_1,
\end{align}
where $[\cU_i,\cU_j]=\cU_i\cU_j-\cU_j\cU_i\in\fA$.

Then the subalgebra $\fA'\subset\fA$ generated over $\bC$ by the elements $\cU_1,\cU_2,\cU_3$ belongs to the invariant subalgebra $\fA^{\!\bG}$.
\end{lemma}


\begin{proof}
Letting $\cC'={\cC}/{\la_{11}}$, we calculate
\begin{gather*}
[\cU_3,\cU_2]=
[\vrho^{xx}\ux_{11}\ux_{21} +\vrho^{xy}\ux_{11}\uy_{21} +\vrho^{yx}\uy_{11}\ux_{21} +\vrho^{yy}\uy_{11}\uy_{21} -{\cC'}(\ux_{11}-\uy_{11}), \ux_{21}+\uy_{21}+\cC]
=\\ =\Sm_1 +\Sm_1^* +\Sm_2 +\Sm_2^* +\Sm_3 +\Sm_3^* +\Sm_4 +\Sm_4^*
\\
\begin{array}{llll}
\text{where}\qquad
&\Sm_1 &=& [\vrho^{xx}\ux_{11}\ux_{21},\uy_{21}] +[\vrho^{xy}\ux_{11}\uy_{21},\ux_{21}] -[{\cC'}\ux_{11},\cC]
\\
&\Sm_2 &=& [\vrho^{xx}\ux_{11}\ux_{21},\ux_{21}] +[\vrho^{xy}\ux_{11}\uy_{21},\uy_{21}],
\\
&\Sm_3 &=& [\vrho^{xx}\ux_{11}\ux_{21},{\cC}] -[{\cC'}\ux_{11},\ux_{21}]
\\
&\Sm_4 &=& [\vrho^{xy}\ux_{11}\uy_{21},{\cC}] -[{\cC'}\ux_{11},\uy_{21}].
\end{array}
\end{gather*}

We obtain $\Sm_1=Q_x\cdot\ux_{11}=-\ux_{11}$ because
\begin{gather*}
\begin{array}{ll}
Q_x & = \vrho^{xx} \ssi_{11}(\bt_{21}) -\ssi\mo_{21}(\vrho^{xx})\ssi\mo_{21}(\muxx_{11,21})\tt_{21}
+\\
& +\vrho^{xy}\ssi_{11}(\tt_{21})-\ssi_{21}(\vrho^{xy})\ssi_{21}(\muxy_{11,21})\bt_{21} -{\cC'}\ssi_{11}(\cC)+{\cC'}{\cC}
\\
& = \vrho^{xx} \left( \ssi_{11}(\bt_{21}) -\tt_{21}\right) +\vrho^{xy}\left(\ssi_{11}(\tt_{21})-\bt_{21}\right)  +(\cC')^2
\\
& = \left( \vrho^{xx}\ssi_{11}(\bt_{21})-\vrho^{xy}\bt_{21} \right) +\left(\vrho^{xy}\ssi_{11}(\tt_{21})-\vrho^{xx}\tt_{21}\right) +(\cC')^2
\\
& = \dfrac{\bt_{21}}{\brg_{21,11}}\ssi_{21}(\la_{21}{-}\bbla_{21})  +\dfrac{\tt_{21}}{\rg_{21,11}}\ssi\mo_{21}(\bbla_{21}{-}\la_{21}) +(\cC')^2=-1.
\end{array}
\end{gather*}

Similarly we get
$Q_y{=}Q_x^*{=}\dfrac{\tt_{21}}{\rg_{21,11}}\ssi\mo_{21}(\bbla_{21}{-}\la_{21})  {+}\dfrac{\bt_{21}}{\brg_{21,11}}\ssi_{21}(\la_{21}{-}\bbla_{21}) {+}(\cC')^2{=}-1.$

Next, $\Sm_2=0$ since
$[\vrho^{xx}\ux_{11}\ux_{21},\ux_{21}]=\left(\vrho^{xx}{-}\ssi_{21}(\vrho^{xx})\muxx_{21,11}\right)\,\ux_{11}\ux^2_{21}=0$
and $[\vrho^{xy}\ux_{11}\uy_{21},\uy_{21}]=0$, similarly.
Finally we have
\begin{gather*}
\Sm_3\!=\![\vrho^{xx}\ux_{11}\ux_{21},{\cC}] -[{\cC'}\ux_{11},\ux_{21}]\!=\!
\left(\vrho^{xx}\ssi_{11}\ssi_{21}(\cC)-\vrho^{xx}\cC -{\cC'}+\ssi_{21}(\cC')\muxx_{21,11}\right) \ux_{11}\ux_{21}=0,
\\
\Sm_4=[\vrho^{xy}\ux_{11}\uy_{21},{\cC}]{-}[{\cC'}\ux_{11},\uy_{21}]
=\big(\vrho^{xy}\ssi_{11}\ssi\mo_{21}(\cC)-\vrho^{xy}{\cC}
-{\cC'}+\ssi\mo_{21}(\cC')\muyx_{21,11}\big)\ux_{11}\uy_{21}=0,
\end{gather*}
which implies $[\cU_3,\cU_2]=-\cU_1$.
\vspace{-0.4cm}
\end{proof}

\vspace{0.1cm}

\begin{lemma}\label{lem_invariant_subring_so}
Let $\so(3;\bC)$ be a special orthogonal Lie algebra standardly presented by the $3{\times}3$ skew-symmetric matrices
$L_x=E_{32}-E_{23}$, $L_y=E_{13}-E_{31}$, $L_z=E_{21}-E_{12}$ with Lie bracket.

The correspondence of the generators
$$L_x\longrightarrow\cU_1,\quad L_y\longrightarrow\cU_1,\quad L_z\longrightarrow\cU_1$$
establishes an isomorphism between the universal enveloping $\op{U}(\so(3;\bC))$ of the special orthogonal Lie algebra $\so(3;\bC)$
and the RTGW algebra $\fA_{\so}=\fA'\subset\fA$.
\end{lemma}

It is a consequence of the relations \eqref{equ_ort_bracket_one} since an algebra $\so(3;\bC)$ is simple.
Theorem \ref{the_so_RTGWA} obviously follows from Lemma \ref{lem_invariant_subring_so}.

\section{Acknowledgments}
J. S. is supported by the FAPESP (2018/18146-5).

\smallskip~

\noindent
N. Golovashchuk: Department of Algebra and Mathematical Logic,
\\
Faculty of Mechanics and Mathematics, Taras Shevchenko National University of Kyiv,
\\
64, Volodymyrska street, 01033 Kyiv, Ukraine

E-mail address: golova@univ.kiev.ua, golovash@gmail.com

Instituto de Matem\'{a}tica e Estat\'{\i}stica,  Universidade  de  S\~{a}o  Paulo, S\~{a}o Paulo SP, Brasil

E-mail address:jfschwarz.0791@gmail.com

\end{document}